\documentclass[11pt, a4paper, english]{amsart}

\usepackage{amsmath, amsthm, amssymb, amsfonts, enumerate}
\usepackage[colorlinks=true,linkcolor=blue,urlcolor=blue]{hyperref}
\usepackage{dsfont}
\usepackage{color}
\usepackage{geometry}
\usepackage{todonotes}

\newtheorem{theorem}{Theorem}[section]
\newtheorem{remark}[theorem]{Remark}
\newtheorem{assumption}[theorem]{Assumption}
\newtheorem{lemma}[theorem]{Lemma}
\newtheorem{proposition}[theorem]{Proposition}

\newtheorem{definition}[theorem]{Definition}

\def \E{\mathsf{E}}

\def \P{\mathsf{P}}

\def \I{\mathcal{I}}
\def \Q{\mathbb{Q}}

\def\X{\widehat{X}}

\def\PP{\widehat{\mathsf{P}}}

\DeclareMathOperator*{\essinf}{ess\,inf}

\title[Public Debt Control under Partial Information]{Optimal Reduction of Public Debt under Partial Observation of the Economic Growth}
\author[Callegaro, Ceci, Ferrari]{Giorgia Callegaro, Claudia Ceci, and Giorgio Ferrari}
\address{G.~Callegaro: Department of Mathematics ``Tullio Levi-Civita'', University of Padova, Via Trieste, 35121 Padova, Italy}
\email{\href{mailto::gcallega@math.unipd.it}{gcallega@math.unipd.it}}
\address{C.~Ceci: Department of Economics, University ``G.\ D'annunzio'' of Chieti-Pescara, Viale Pindaro 42, I-65127 Pescara, Italy}
\email{\href{mailto: claudia.ceci@unich.it}{claudia.ceci@unich.it}}
\address{G.~Ferrari: Center for Mathematical Economics (IMW), Bielefeld University, Universit\"atsstrasse 25, 33615 Bielefeld, Germany}
\email{\href{mailto:giorgio.ferrari@uni-bielefeld.de}{giorgio.ferrari@uni-bielefeld.de}}
\date{\today}

\numberwithin{equation}{section}

\begin{document}

\begin{abstract} 
We consider a government that aims at reducing the debt-to-gross domestic product (GDP) ratio of a country. The government observes the level of the debt-to-GDP ratio and an indicator of the state of the economy, but does not directly observe the development of the underlying macroeconomic conditions. The government's criterion is to minimize the sum of the total expected costs of holding debt and of debt's reduction policies. We model this problem as a singular stochastic control problem under partial observation. The contribution of the paper is twofold. Firstly, we provide a general formulation of the model in which the level of debt-to-GDP ratio and the value of the macroeconomic indicator evolve as a diffusion and a jump-diffusion, respectively, with coefficients depending on the regimes of the economy. These are described through a finite-state continuous-time Markov chain. We reduce via filtering techniques the original problem to an equivalent one with full information (the so-called separated problem), and we provide a general verification result in terms of a related optimal stopping problem under full information. Secondly, we specialize to a case study in which the economy faces only two regimes, and the macroeconomic indicator has a suitable diffusive dynamics. In this setting we provide the optimal debt reduction policy. This is given in terms of the continuous free boundary arising in an auxiliary fully two-dimensional optimal stopping problem. 
\end{abstract}

\maketitle

\smallskip


\smallskip

{\textbf{MSC2010 subject classification}}: 93E20, 60G35, 93E11, 60G40, 60J60, 91B64.

%


\section{Introduction}

The question of optimally managing debt-to-GDP ratio (also called ``debt ratio'') of a country has become particularly important in the latest years. Indeed, concurrently with the financial crisis started in 2007, debt-to-GDP ratio exploded from an average of 53\% to circa 80\% in developed countries. Clearly, the debt management policy of a government highly depends on the underlying macroeconomic conditions; indeed, these affect, for example, the growth rate of GDP which, in turn, determines the growth rate of the debt-to-GDP ratio of a country. However, in practice it is typically neither possible to measure in real-time the growth rate of GDP, nor one can directly observe the underlying business cycles. On August 24, 2018, Jerome H.\ Powell, Chairman of the Federal Reserve, said:\footnote{Speech at ``Changing Market Structure and Implications for Monetary Policy'', a symposium sponsored by the Federal Reserve Bank of Kansas City in Jackson Hole, Wyoming.} 
\vspace{0.1cm}

\emph{...In conventional models of the economy, major economic quantities such as inflation, unemployment, and the growth rate of gross domestic product fluctuate around values that are considered ``normal'', or ``natural'' or ``desired''. The FMOC (Federal Open Market Committee) has chosen a 2 percent inflation objective as one of these desired values. \textbf{The other values are not directly observed, nor can they be chosen by anyone}...}
\vspace{0.1cm}

Following an idea that dates back to \cite{Hamilton}, in this paper we suppose that the GDP growth rate of a country is modulated by a continuous-time Markov chain that is not directly observable. The Markov chain has $Q\geq2$ states modeling the different business cycles of the economy, so that a shift in the macroeconomic conditions induces a change in the value of the growth rate of GDP. The government can observe only the current levels of the debt-to-GDP ratio and of a macroeconomic indicator. The latter might be, e.g., one of the so-called ``Big Four''\footnote{These indicators constitute the Conference Board's Index of Coincident Indicators; they are employment in non agricultural businesses, industrial production, real personal income less transfers, and real manufacturing and trade sales. We refer to, e.g., \cite{SW2003}, where the authors present a wide range of economic indicators and examine the forecasting performance of various of them in the recession of 2001.}, which are usually considered proxies of the industrial production index, hence of the business conditions. 

The government may intervene in order to decrease the level of the debt ratio, e.g.\ through fiscal policies or imposing austerity policies in the form of spending cuts. We assume that the debt ratio is instantaneously affected by any such policy. Debt reductions must not necessarily be performed at rates, but also lump sum actions are allowed, and the cumulative amount of debt ratio's decrease is the government's control variable. Any decrease of the debt ratio results in proportional costs, and the government aims at choosing a debt-reduction policy that minimizes the total expected costs of holding debt, plus the total expected costs of interventions on the debt ratio. 
In line with recent papers on stochastic control methods for optimal debt management (see \cite{CadAgui}, \cite{CadAgui2}, \cite{Ferrari18} and \cite{FeRo}), we model the previous problem as a singular stochastic control problem. However, differently to all the previous works, our problem is formulated in a partial observation setting, thus leading to a completely different mathematical analysis. In our model, the observations consist of the debt ratio and of the macroeconomic indicator. The 
debt ratio is a linearly controlled geometric Brownian motion, and its drift is given in terms of the GDP growth rate, which is modulated by the unobservable continuous-time Markov chain $Z$. The macroeconomic indicator is a real-valued jump-diffusion which is correlated to the debt ratio process, and which has drift, and both intensity and jump sizes, depending on $Z$.
\vspace{0.15cm}

\textbf{Our Contributions.} Our study of the optimal debt reduction problem is performed thought three main steps. 

First of all, via advanced filtering techniques with mixed-type observations, we reduce the original problem to an equivalent problem under full information, the so-called separated problem. The filtering problem consists in characterizing the conditional distribution of the unobservable Markov chain $Z$, at any time $t$, given observations up to time $t$. The case of diffusion observations has been widely studied in literature and textbook treatments can be found in \cite{Ell},  \cite{Kall}, and \cite{LipShir}. There are also known results for pure-jump observations (see, e.g., \cite{Bremaud}, \cite{CeciGerardi00}, \cite{CeciGerardi06}, \cite{KKM}, and references therein). More recently, filtering problems with mixed-type information, which involve pure-jump processes and diffusions, have been studied in \cite{ColaneriCeci2012}, \cite{ColaneriCeci2014}, and \cite{Frey}.  Due to the structure of our observations' dynamics we cannot apply the probability reference method (see \cite{ColaneriCeci2014} and \cite{Zakai}), and for this reason we choose an alternative route based on the innovation approach, which leads to the Kushner-Stratonovich equation. Moreover, differently to \cite{ColaneriCeci2012} and \cite{Frey}, in our framework the innovation process is two-dimensional and, therefore, the innovation method employed in these papers must be suitably adapted to our context. By showing that the Kushner-Stratonovich equation admits a unique strong solution, we are then able to prove that the original problem under partial observation and the separated problem are equivalent in the sense that they share the same value and the same optimal control.


Secondly, we exploit the convex structure of the separated problem, and we provide a general probabilistic verification theorem. This result - which is in line with findings in \cite{BK}, \cite{DeAFeFe} and \cite{Ferrari18}, among others - relates the optimal control process to the solution to an auxiliary optimal stopping problem. Moreover, it proves that the value function of the separated problem is the integral - with respect to the controlled state variable - of the value function of the optimal stopping problem. The stopping problem thus gives the optimal timing at which one additional unit of debt should be reduced.

Finally, by specifying a setting in which the continuous-time Markov chain faces only two regimes (a fast growth or slow growth phase) and the macroeconomic indicator is a suitable diffusion process, we are able to characterize the optimal debt reduction policy. In this framework, the filter process is a two-dimensional process $(\pi_t,1-\pi_t)_{t\geq0}$, where $\pi_t$ is the conditional probability at time $t$ that the economy enjoys the fast growth phase. We prove that the optimal control prescribes to keep at any time the debt ratio below an endogenously determined curve that is a function of government's belief about the current state of the economy. Such a debt ceiling is the free boundary of the fully two-dimensional optimal stopping problem that is related to the separated problem in the sense of the previously discussed verification theorem. By using almost exclusively probabilistic means, we are able to show that the value function of the auxiliary optimal stopping problem is a $C^1$-function of its arguments, and thus enjoys the so-called smooth-fit property. Moreover, the free boundary is a continuous, bounded, and increasing function of the filter process. This last monotonicity property has also a clear economic interpretation: the more the government believes that the economy enjoys a regime of fast growth, the less strict the optimal debt reduction policy should be. 

As a remarkable byproduct of the regularity of the value function of the optimal stopping problem, we also obtain that the value function of the singular stochastic control problem is a classical solution to its associated Hamilton-Jacobi-Bellman (HJB) equation. The latter takes the form of a variational inequality involving an elliptic second-order partial differential equation (PDE). It is worth noticing that the $C^2$ regularity of the value function implies the validity of a second-order principle of smooth fit, usually observed in one-dimensional problems.

We believe that the study of the auxiliary fully two-dimensional optimal stopping problem is a valuable contribution to the literature on its own. Indeed, if the literature on one-dimensional optimal stopping problems is very rich, the problem of characterizing the optimal stopping rule in multi-dimensional settings has been so far rarely explored in the literature (see the recent \cite{CCMS}, \cite{DeAFeFe} and \cite{JP2017}, among the very few papers dealing with multi-dimensional stopping problems). This discrepancy is due to the fact that a standard guess-and-verify approach, based on the construction of an explicit solution to the variational inequality arising in the considered optimal stopping problem, is not anymore applicable in multi-dimensional settings where the variational inequality involves a PDE rather than an ordinary differential equation.
\vspace{0.15cm}

\textbf{Related Literature.} As already noticed above, our paper is placed among those recent works addressing the problem of optimal debt management via continuous-time stochastic control techniques. In particular, \cite{CadAgui} and \cite{CadAgui2} model an optimal debt reduction problem as a one-dimensional control problem with singular and bounded-velocity controls, respectively. In \cite{FeRo} the government is allowed to increase and decrease the current level of debt ratio, and the interest rate on debt is modulated by a continuous-time observable Markov chain. The mathematical formulation leads to a one-dimensional bounded-variation stochastic control problem with regime switching. In \cite{Ferrari18}, when optimally reducing the debt ratio, the government takes into consideration the evolution of the inflation rate of the country. The latter evolves as an uncontrolled diffusion process and affects the growth rate of the debt ratio, which is a process of bounded variation. In this setting, the debt reduction problem is formulated as a two-dimensional singular stochastic control problem whose HJB equation involves a second-order linear parabolic partial differential equation. All the previous papers are formulated in a full information setting, while ours is under partial observation.

The literature on singular stochastic control problems under partial observation is also still quite limited. Theoretical results on the PDE characterization of the value function of a two-dimensional optimal correction problem under partial observation are obtained in \cite{MenaldiRobin83}, whereas a general maximum principle for a not necessarily Markovian singular stochastic control problem under partial information has been more recently derived in \cite{OksendalSulem}. We also refer to \cite{DeA18} and \cite{DV}, where it is provided a thorough study of the optimal dividend strategy in models in which the surplus process evolves as drifted Brownian motion with unknown drift that can take only two constant values, with given probability. 
\vspace{0.15cm}

\textbf{Outline of the Paper.} The rest of the paper is organized as follows. In Section \ref{sec:problem} we introduce the setting and formulate the problem. The reduction of the problem under partial observation to the separated problem is performed in Section \ref{sec:reduction}; in particular, the filtering results are presented in Section \ref{sec:filtering}. The probabilistic verification theorem connecting the separated problem to one of optimal stopping is then proved in Section \ref{sec:verificationthm}. In Section \ref{sec:casestudy} we then consider a case study in which the economy faces only two regimes. Its solution, presented in Sections \ref{sec:optcontr} and \ref{sec:regularityV}, hinges on the study of a two-dimensional optimal stopping problem that is performed in Section \ref{sec:relatedOS}. Finally, Appendix \ref{app:proofs} collects the proofs of some technical filtering results.


\section{Setting and Problem Formulation}
\label{sec:problem}

\subsection{The Setting}
\label{sec:setting}

Consider the complete filtered probability space $(\Omega,\mathcal F,(\mathcal F_t)_{t \ge 0},\P)$, capturing all the uncertainty of our setting. Here, $\mathbb{F}:={(\mathcal F_t)}_{t \ge 0}$ denotes the \emph{full information filtration}. We suppose that such a filtration satisfies the usual hypotheses of completeness and right-continuity. 

We denote by $Z$ a continuous-time finite-state Markov chain describing the different states of the economy. For $Q\geq 2$, let $S := \{1, 2,\dots, Q\}$ be the state space of $Z$ and $\{\lambda_{ij}\}_{1 \leq i,j \leq Q}$ its generator matrix. Here, $\lambda_{ij}$, $i \neq j$, gives the intensity of a transition from state $i$ to state $j$, and it is such that $\lambda_{ij} \geq 0$, for $i \neq j$, and $\sum_{j=1, j\ne i}^Q \lambda_{ij} = - \lambda_{ii}$. For any time $t\geq0$, $Z_t$ is $\mathcal{F}_t$-measurable.

In absence of any intervention by the government, we assume that the (uncontrolled) debt-to-GDP ratio evolves as
\begin{equation}
\label{Xzero}
d X^0_t = \big(r - g(Z_t)\big) X^0_t dt + \sigma X^0_t dW_t , \quad X^0_0=x \in (0,\infty),
\end{equation}
where $W$ is a standard $\mathbb{F}$-Brownian motion on $(\Omega, \mathcal{F})$ independent of $Z$, $r\geq0$ and $\sigma>0$ are constants, and $g: S \rightarrow \mathbb R$. The constant $r$ is the real interest rate on debt, $\sigma$ is the debt's volatility, and $g(i) \in \mathbb{R}$ is the rate of the GDP's growth when the economy is in state $i\in S$.  

It is clear that equation \eqref{Xzero} admits a unique strong solution, and, when needed, for any $x>0$ we shall denote it by $X^{x,0}$. The current level of the debt-to-GDP ratio is known to the government at any time $t$, and $X^{x,0}$ is therefore the first component of the so-called \emph{observation process}.

The government also observes a macroeconomic stochastic indicator $\eta$, e.g.\ one of the so-called ``Big Four'', which we interpret as a proxy of the business conditions. We assume that $\eta$ is a jump-diffusion process solving the stochastic differential equation
\begin{equation}
\label{eta}
d \eta_t= b_1(\eta_t, Z_t)  dt + \sigma_1(\eta_t) dW_t + \sigma_2(\eta_t) dB_t + c(\eta_{t^-},  Z_{t^-}) dN_t,  \qquad \eta_{0}= q \in  \I,
\end{equation}
where $b_1$,  $c$, $\sigma_1 >0$, and $\sigma_2>0$ are measurable functions of their arguments, and $\I \subseteq \mathbb R$ is the state space of $\eta$. Here, $B$ is an $\mathbb{F}$-standard Brownian motion, independent of $W$ and $Z$. Moreover, $N$ is an $\mathbb{F}$-adapted point process, without common jump times with $Z$, independent of $W$ and $B$. The predictable intensity of $N$ is denoted by $\{\lambda^N(Z_{t^-})\}_{t\geq 0}$ and depends on the current state of the economy, with $\lambda^N(\,\cdot\,)>0$ being a measurable function. From now on, we assume the following assumptions that ensure strong existence and uniqueness of the solution to equation \eqref{eta} (see \cite{GS}, among others). 

\begin{assumption}
\label{ass:eta}
The functions $b_1: \I \times S \to \mathbb{R}$, $\sigma_1: \I \to (0,\infty)$, $\sigma_2: \I \to (0,\infty)$, and $c: \I \times S \to \mathbb{R}$ are such that for any $i \in S$:
\begin{itemize}
\item[(i)] (Continuity) $b_1(\cdot, i)$, $\sigma_1(\cdot)$, $\sigma_2(\cdot)$ and $c(\cdot, i)$ are continuous;
\item[(ii)] (Local Lipschitz conditions) for any $R>0$, there exists a constant $L_R>0$ such that if $|q|< R, |q'|<R$, $q, q' \in \I$, then
$$|b_1(q, i) - b_1(q', i)| + |\sigma_1(q)- \sigma_1(q')| + |\sigma_2(q)- \sigma_2(q')| + |c(q, i) - c(q', i)| \leq L_R|q-q'|;$$
\item[(iii)] (Growth conditions) there exists a constant $C>0$ such that  
$$|b_1(q, i)|^2 +  |\sigma_1(q)|^2 +  |\sigma_2(q)|^2 + |c(q, i)|^2 \leq C (1 + |q|^2).$$
\end{itemize}
\end{assumption}

The dynamics proposed in equation \eqref{eta} is of a jump-diffusive type, and it allows for a size and intensity of the jumps affected by the state of the economy. It is therefore flexible enough to describe a large class of stochastic factors which may exhibit jumps. 

The \emph{observation filtration} $\mathbb H = {(\mathcal H_t)}_{t \ge 0}$ is defined as 
\begin{equation}
\label{FiltrationH}
\mathbb H := \mathbb F^{X^0} \vee \mathbb F^\eta, 
\end{equation}
where $\mathbb F^{X^0}$ and $\mathbb F^\eta$ denote the natural filtrations generated by $X^0$ and $\eta$, respectively, as usual augmented by $\P$-null sets. Clearly, $(X^0, \eta)$ is adapted to both $\mathbb H$ and $\mathbb F$, and 
$$\mathbb H \subset \mathbb F.$$
The above inclusion means that the government cannot directly observe the state of the economy $Z$, but this has to be inferred through the observation of 
$(X^0, \eta)$. We are therefore working in a \emph{partial information setting}.


\subsection{The Optimal Debt Reduction Problem}
\label{sec:Problem-P1}

The government can reduce the level of the debt-to-GDP ratio by intervening on the primary budget balance (i.e.\ the overall difference between government revenues and spending), for example through austerity policies in the form of spending cuts. By doing so the debt ratio dynamics modifies as 
 \begin{equation}
\label{eq:dynX}
dX^{\nu}_t= \big(r - g(Z_t)\big) X^{\nu}_t dt + \sigma X^{\nu}_t dW_t - d \nu_t, \quad X^{\nu}_{0^-}=x >0.
\end{equation}
The process $\nu$ is the control that the government chooses based on the information at its disposal. Precisely, $\nu_t$ defines the cumulative reduction of the debt-to-GDP ratio made by the government up to time $t$, and $\nu$ is therefore a nondecreasing process belonging to the set
\begin{align*}
& \mathcal M(x, \underline y,q):=\Big\{\nu:\Omega \times \mathbb R_+ \rightarrow \mathbb R_+ : {(\nu_t(\omega) := \nu (\omega,t))}_{t \ge 0} \ \textrm{is nondecreasing, right-continuous,} \nonumber \\
& \hspace{0.3cm} \mathbb H-\textrm{adapted, such that}\, X_t^{\nu} \ge 0 \ \textrm{for every} \ t \ge 0,\ X^{\nu}_{0^-}=x,\,\, \P(Z_0=i)=y_i,\,\, i \in S,\,\, \eta_0=q \ \textrm{a.s.} \Big\},
\end{align*}
for any given and fixed $x \in (0, \infty)$ initial value of $X^{\nu}$, $q \in  \I$ initial value of $\eta$, and ${\underline y} \in \mathcal{Y}$. Here  
$$\mathcal{Y} :=  \Big\{{\underline y} = (y_1, \dots y_Q): y_i \in [0,1], \,\,\,  i=1, \dots Q,  \,\,\, \sum_{i=1}^Q  y_i =1 \Big\},$$
is the probability simplex on $\mathbb R^Q$, representing the space of initial distributions of the process $Z$. From now on, we set $\nu_{0^-}=0$ a.s.\ for any $\nu \in \mathcal M(x, \underline y,q)$. 

\begin{remark}
\label{rem:notation}
Notice that in the definition of the set $\mathcal{M}$ above, as well as in \eqref{costfunctional} and in {\textbf{(P1)}} below, we have stressed the dependency on the initial data $(x,\underline{y},q)$ just for notational convenience, and not to stress any Markovian nature of the considered problem, which is in fact not such.
\end{remark}

For any $(x,\underline{y},q) \in (0,\infty) \times \mathcal{Y} \times \mathcal{I}$ and $\nu \in \mathcal{M}(x,\underline y,q)$, there exists a unique solution to \eqref{eq:dynX}, denoted by  $X_t^{x,\nu}$, that is given by
\begin{equation}
\label{soldynX}
X_t^{x,\nu} = X^{1,0}_t\bigg[x - \int_0^t \frac{d\nu_s}{X^{1,0}_s}\bigg], \qquad t \geq 0, \quad X_{0^-}^{x,\nu}=x,
\end{equation}
where 
$$X^{1,0}_t = \displaystyle e^{\int_0^t (r-g(Z_s)) ds - {\frac{1}{2}} \sigma^2 t  + \sigma W_t}, \qquad t \geq 0.$$  

Here, and in the rest of this paper, we shall use the notation $\int_0^t (\,\cdot\,)d\nu_s = \int_{[0,t]} (\,\cdot\,) d\nu_s$ for the Lebesgue-Stieltjes integral with respect to the random measure $d\nu_{\cdot}$ induced by the nondecreasing process $\nu$ on $[0,\infty)$.

\begin{remark}
\label{rem:micro}
The dynamics \eqref{eq:dynX} might be justified in the following way. Suppose that the public debt (in real terms), $D$, and the GDP, $Y$, follow the classical dynamics
\begin{equation*}
\left\{
\begin{array}{rcl}
dD_t &=& r D_t dt - d \xi_t, \qquad \quad \,\,\quad D_{0^-}=d > 0,\\
dY_t&=& g(Z_t) Y_t dt + \sigma Y_t dW_t, \quad Y_0= y>0,\\
\end{array}
\right.
\end{equation*}
where $\xi_t$ is the cumulative real budget balance up to time $t$. An easy application of It\^o's formula then gives that the ratio $X:=D/Y$ evolves as in \eqref{eq:dynX}, upon setting $d\nu:=d\xi/Y$. 
\end{remark}

The government aims at reducing the level of the debt ratio. Having a level of debt ratio $X_t=x$ at time $t\geq 0$ when the state of the economy is $Z_t=i$, the government incurs  an instantaneous cost $h(x,i)$. This cost may be interpreted as a measure of the resulting losses for the country due to the debt, as, e.g., a tendency to suffer low subsequent growth (see \cite{ECB13} and \cite{WK15}, among others). The cost function $h:\mathbb{R} \times S \mapsto \mathbb{R}_+$ fulfills the following requirements (see also \cite{CadAgui} and \cite{Ferrari18})
\begin{assumption}
\label{ass:h}
\begin{itemize}\hspace{10cm}
\item[(i)] For any $i\in S$, the mapping $x \mapsto h(x,i)$ is strictly convex, continuously differentiable, and it is nondecreasing on $\mathbb{R}_+$. Moreover, $h(0,i)=0$; 
\item[(ii)] For any given $x\in (0,\infty)$ and $i\in S$ one has  
$$
\E\bigg[\int_0^{\infty} e^{-\rho t} h\big(X_t^{x,0}, i \big) dt\bigg] + \E\bigg[\int_0^{\infty} e^{-\rho t} X_t^{1,0} h_x \big(X_t^{x,0}, i \big) dt\bigg] < \infty.
$$
\end{itemize}
\end{assumption}
A quadratic cost function of the form $h(x,i) = \frac{1}{2}\vartheta_i x^2$, $(x,i) \in [0,\infty)\times S$, $\vartheta_i>0$, clearly satisfies Assumption \ref{ass:h} for a suitable $\rho>0$. 

Whenever the government intervenes in order to reduce the debt-to-GDP ratio, it incurs a proportional cost. We assume that the marginal cost of each intervention is normalized to one.

Given an intertemporal discount rate $\rho>0$, for any given and fixed $(x,\underline y,q) \in (0,\infty) \times \mathcal{Y} \times \I$, the government thus aims at minimizing the expected total cost functional
\begin{equation}
\label{costfunctional}
\mathcal{J}_{x, \underline y,q}(\nu) := \E_{(x,\underline y ,q)}\bigg[\int_0^{\infty} e^{-\rho t} h\big(X_t^{x,\nu}, Z_t\big) dt  + \int_0^{\infty} e^{-\rho t} d \nu_t\bigg], \quad \nu \in \mathcal{M}(x, \underline y, q).
\end{equation}
Here $\E_{(x, \underline y,q)}$ is the expectation under the condition that $X^{x,\nu}_{0^-}=x$, $Z$ has initial distribution $\underline y$, and 
$\eta_{0}=q$. The government's problem under partial observation can be therefore defined as
\begin{equation*}
{\textbf{(P1)}} \qquad V_{po}(x,\underline y ,q):= \inf_{\nu \in \mathcal M(x,\underline y,q)} \mathcal{J}_{x,\underline y,q}(\nu), \quad (x,\underline y,q) \in (0,\infty) \times \mathcal{Y} \times \I.
\end{equation*}

One has that $V_{po}$ is well defined and finite. Indeed, it is nonnegative, due to the nonnegativity of $h$; moreover, since the admissible policy ``instantaneously reduce at initial time the debt ratio to $0$'' is a priori suboptimal and it has cost $x$, then $V_{po} \leq x$.

We would like to stress once more that any $\nu \in \mathcal M(x,\underline y,q)$ is $\mathbb H$-adapted, and therefore Problem \textbf{(P1)} is a stochastic optimal control problem under partial observation. In particular, it is a \emph{singular stochastic control problem under partial observation}; that is, an optimal control problem in which the random measures induced by the nondecreasing control processes on $[0,\infty)$ might be singular with respect to the Lebesgue measure, and in which one component of the state variable, $Z$, is not directly observable by the controller.  

In its current formulation, the optimal debt reduction problem is not Markovian and it is therefore not directly solvable by standard means of stochastic control theory. In the next section, by using techniques from filtering theory, we will introduce an equivalent problem under complete information, the so-called \emph{separated problem}. This will enjoy a Markovian structure and its solution will be characterized in Section \ref{sec:verificationthm} through a Markovian optimal stopping problem.


\section{Reduction to an Equivalent Problem under Complete Information}
\label{sec:reduction}

In this section we derive the \emph{separated problem}. To this end, we first study the filtering problem arising in our model. As already discussed in the introduction, results on such a filtering problem cannot be directly obtained from existing literature due to the structure of our dynamics.

\subsection{The Filtering Problem}
\label{sec:filtering}

The filtering problem consists in finding the best-mean squared estimate of $f(Z_t)$, for any $t$ and any measurable function $f$, on the basis of the information collected up to time $t$. In our setting, such an information flow is given by the filtration $\mathbb{H}$. That estimate can be described through the filter process ${(\pi_t)}_{t \ge 0}$, which provides the conditional distribution of $Z_t$ given $\mathcal{H}_t$, for any time $t$ (see, for instance, \cite{LipShir}). It is defined as the $\mathbb H$-c\`{a}dl\`{a}g (right-continuous with left limits) process taking values in the space of probability measures on $S = \{1, \dots, Q\}$ such that
\begin{equation}
\label{pi1}
\pi_t(f) := \E\big[f(Z_t)\big|\mathcal{H}_t\big],
\end{equation}
for all measurable functions $f$ on $S$.  
Since $Z$ takes only a finite number of values, the filter is completely described by the vector 
\begin{equation}
\pi_t(f_i) =\P(Z_t = i | \mathcal H_t), \quad i \in S,
\end{equation}
where $f_i(z):= \mathds{1}_{\{z=i\}}$, $i \in S$. With a slight abuse of notation, in the following we will denote by $\pi(i)$ the process 
$\pi(f_i)$, so that for all measurable functions $f$ we have from \eqref{pi1}, $\pi_t(f) = \sum_{i=1}^Q f(i) \pi_t(i).$

Setting $\beta(Z_t):=r-g(Z_t)$ and, accordingly, $\beta_i:= r - g(i)$, $i \in S$, notice that $\beta$ is clearly a bounded function. Then, define the two processes $I$ and $I^1$ such that for any $t\geq0$
\begin{equation} \label{Inn}
I_t := W_t - \int_0^t \sigma^{-1}  \big( \pi_s(\beta) - \beta(Z_s) \big) ds, \quad I^1_t := B_t - \int_0^t \big( \pi_s(\alpha(\eta_s,\,\cdot\,)) - \alpha(\eta_s, Z_s) \big) ds,
\end{equation}
where
\begin{equation} 
\label{alfa}
\alpha(q, i):= \sigma_2(q)^{-1}  \Big\{ b_1(q, i)  - \sigma^{-1} \beta(i) \sigma_1(q)\Big\}, \quad (q,i) \in \I \times S.
\end{equation}

Henceforth, we will work under the following Novikov's condition. 
\begin{assumption}
\label{ass:Novikov}
\begin{equation}
\label{Novikov}
\displaystyle  \E\left[e^{\frac{1}{2}\int_0^t  \alpha^2(\eta_s, Z_s) ds} \right]< \infty, \quad \text{for any}\,\,t\geq0.
\end{equation}
\end{assumption}

Under Assumption \ref{ass:Novikov}, by classical results from filtering theory (see, e.g., \cite{LipShir}), the \emph{innovation processes} $I$ and $I^1$ are Brownian motions with respect to the filtration $\mathbb{H}$. Moreover, given the assumed independence of $B$ and $W$, they turn out to be independent.  

The integer-valued random measure associated to the jumps of $\eta$ is defined as
\begin{equation} 
\label{measure-m}
m(dt, dq):= \sum_{s: \Delta \eta_s \neq 0} \delta_{( s,  \Delta \eta_s  )} (ds, dq),
\end{equation}
where $\delta_{(a_1,a_2)}$ denotes the Dirac measure at point $(a_1,a_2) \in \mathbb{R}_+ \times \mathbb{R}$. Notice that the $\mathbb H$-adapted random measure $m$ is such that
\begin{equation} 
\int_0^t c(\eta_{s^-}, Z_{s^-}) \mathds{1}_{\{ c(\eta_{s^-}, Z_{s^-}) \neq 0\}} dN_s = \int_0^t \int_{\mathbb R} q\, m(ds,dq), \quad t>0.
\end{equation}

To proceed further we need the following useful definitions.

\begin{definition}
\label{def:predictproj}
($\mathbb G$-Predictable Process indexed by $\mathbb{R}$). Given any filtration $\mathbb G$, let  ${\mathcal  P}(\mathbb G)$ denote the predictable $\sigma$-field on $(0,\infty)\times \Omega$ and ${\mathcal  B}(\mathbb R)$ the Borel $\sigma$-algebra on $\mathbb R$. Any mapping $H: (0,\infty) \times \Omega\times \mathbb R \to \mathbb R$ which is ${\mathcal P}(\mathbb G) \times {\mathcal  B}(\mathbb R)$-measurable is called {\rm{$\mathbb G$-predictable process indexed by $\mathbb R$}}.
\end{definition}  

Letting \begin{equation} \label{fm}{\mathcal F}_t^m := \sigma \{m((0,s] \times A): 0 \leq s \leq t, A \in {\mathcal B}( \mathbb R)\},\end{equation}
we denote by $\mathbb F^m:=(\mathcal{F}^m_t)_{t\geq0}$ the filtration generated by the random measure $m(dt,dq)$.

\begin{definition}
\label{def:dualmeasure}
(Dual Predictable Projection of $m$). Given any filtration $\mathbb G$, such that $\mathbb F^m \subseteq  \mathbb G$, the {\rm{$\mathbb G$-dual predictable projection of $m$}}, denoted by $m^{p, \mathbb G}(dt, dq)$, is defined as the unique positive $\mathbb G$-predictable random measure such that for any nonnegative, $\mathbb G$-predictable process $\Phi$ indexed by $\mathbb R$
\begin{equation} 
\label{pro}
\E\bigg[\int_0^\infty \int_{\mathbb R} \Phi(s,q)   m(ds,dq)\bigg]  = \E\bigg[\int_0^\infty \int_{\mathbb R} \Phi(s,q) m^{p, \mathbb G}(dt, dq)\bigg].
\end{equation}
\end{definition}

To prove that a positive $\mathbb G$-predictable random measure provides the $\mathbb G$-dual predictable projection of $m$ it suffices to prove that equation \eqref{pro} holds true for any process of the form $\Phi(t,q) = C_t \mathds{1}_A(q)$ with $C$ a nonnegative $\mathbb G$-predictable process and $A \in {\mathcal  B}(\mathbb R)$. For further details we refer to \cite{Bremaud} and \cite{Jacod}.

We now aim at deriving an equation for the evolution of the filter (the \emph{filtering equation}). To this end we use the so-called \emph{innovation approach} (see \cite{Bremaud}, \cite{ColaneriCeci2012}, and \cite{LipShir}, among others), which, in our setting, requires the introduction of the $\mathbb H$-compensated jump measure of $\eta$ 
\begin{equation}
\label{mpi}
m^{\pi}(dt, dq):= m(dt, dq) - m^{p, \mathbb H}(dt, dq),
\end{equation}
where $m^{p, \mathbb H}(dt, dq)$ is the $\mathbb H$-dual predictable projection of $m$ (cf.\ Definition \ref{def:dualmeasure} above). The triplet $(I, I^1, m^{\pi})$ also represents the building block of the construction of $\mathbb H$-martingales, as it is shown in Proposition \ref{rappre} below. We start determining the form of $m^{p, \mathbb H}$. 

\begin{proposition}
\label{prop:dualpredictable}
The $\mathbb H$-dual predictable projection of $m$ is given by 
\begin{align} \label{dual}
 m^{p, \mathbb H}(dt, dq)= \sum_{i=1}^Q\pi_{t^-}(i) \lambda^N(i) \mathds{1}_{\{ c(\eta_{t^-}, i) \neq 0\}} \delta_{c(\eta_{t^-}, i)} (dq) dt,
\end{align}
where $\delta_{a}$ denotes the Dirac measure at point $a \in \mathbb{R}$.
\end{proposition}

\begin{proof}
\emph{Step 1.} First, we prove that the $\mathbb F$-dual predictable projection of $m$ is given by
\begin{equation}
\label{mp}
m^{p, \mathbb F}(dt,dq) := \lambda^N(Z_{t^-}) \mathds{1}_{\{ c(\eta_{t^-}, Z_{t^-}) \neq 0\}}\delta_{c(\eta_{t^-},  Z_{t^-})} (dq) dt.
\end{equation}

Let $A\in {\mathcal  B}(\mathbb R)$ and introduce
\begin{equation}
\label{callN}
\mathcal{N}_t(A):= m((0,t] \times A) = \sum_{s\leq t} \mathds{1}_{\{\Delta \eta_s\in A \setminus \{0\}\}}, \qquad t \geq 0.
\end{equation}
$\mathcal{N}(A)$ is the point process counting the number of jumps of $\eta$ up to time $t$ with jumps' size in the set $A$.  
Since by \eqref{eta} one has that $\Delta \eta_s =  c(\eta_{s^-}, Z_{s^-}) \mathds{1}_{\{ c(\eta_{s^-}, Z_{s^-}) \neq 0\}} \Delta N_s$,  $\forall s\geq 0$, and $N$ is a point process with $\mathbb F$-predictable intensity given by  $\{\lambda^N(Z_{t^-})\}_{t\geq 0}$, we obtain that for each $C$ nonnegative $\mathbb F$-predictable process
\begin{align*}
& \E\bigg[\int_0^t \!\!C_s \, d\mathcal{N}_s(A) \bigg]  = \E\bigg[\int_0^t \!\!\! C_s\mathds{1}_{\{ c(\eta_{s^-}, Z_{s^-}) \in A \setminus \{0\}\}} dN_s\bigg] = \E\bigg[ \int_0^t \!\!\! C_s \mathds{1}_{\{ c(\eta_{s^-}, Z_{s^-}) \in A \setminus \{0\}\}} \lambda^N(Z_{s^-}) ds\bigg].
\end{align*}
That is, for any $A\in {\mathcal  B}(\mathbb R)$, we have that $\{\lambda^N(Z_{t^-})  \mathds{1}_{\{ c(\eta_{t^-}, Z_{t^-}) \in A \setminus \{0\}\}}\}_{t\geq 0}$
provides the $\mathbb F$-predictable intensity of the counting process $\mathcal{N}(A)$. Recalling \eqref{callN} and Definition \ref{def:dualmeasure}, this implies that $m^{p, \mathbb F}(dt,dq)$ given in \eqref{mp} coincides with the $\mathbb F$-dual predictable projection of $m$, since equation \eqref{pro} holds with the choice $\mathbb G = \mathbb F$ and $\Phi(t,q) = C_t \mathds{1}_A(q)$, with $C$ an arbitrary nonnegative $\mathbb F$-predictable process and $A\in {\mathcal  B}(\mathbb R)$. 
\vspace{0.10cm}

\emph{Step 2.} As in Proposition 2.3 in \cite{Ceci2012} we can now derive the $\mathbb H$-dual predictable projection of $m^{p, \mathbb F}$, denoted by $m^{p, \mathbb H}(dt, dq)$, by simply projecting $m^{p, \mathbb F}$ with respect to the observation flow $\mathbb H$. Precisely, we have that the $\mathbb H$-predictable intensity of the point process $\mathcal{N}(A)$, $\forall A\in {\mathcal  B}(\mathbb R)$, is given by 
$$\pi_{t^-} (\lambda^N(.)  \mathds{1}_{\{ c(\eta_{t^-}, .) \in A \setminus \{0\}\}}) = \sum_{i=1}^Q\pi_{t^-}(i) \lambda^N(i) \mathds{1}_{\{ c(\eta_{t^-}, i) \in A \setminus \{0\}\}}, \quad \forall A\in {\mathcal  B}(\mathbb R).$$
This implies that  $m^{p, \mathbb H}(dt, dq)$ is given by \eqref{dual}, since \eqref{pro} holds with the choice $\mathbb G = \mathbb H$, $\Phi(t,q) = C_t \mathds{1}_A(q)$, with $A\in {\mathcal  B}(\mathbb R)$ and $C$ an arbitrary nonnegative $\mathbb H$-predictable process.
\end{proof}
 
An essential tool to prove that the original problem under partial information is equivalent to the separated one is the characterization of the filter as the unique solution to the filtering equation (see \cite{CeciGerardi98} and \cite{ElKaroui}). In order to derive the filtering equation solved by $\pi$ we first give a representation theorem for $\mathbb H$-martingales. The proof of the following technical result is given in Appendix \ref{app:proofs}.

\begin{proposition}
\label{rappre}
Under Assumptions \ref{ass:eta}  and \ref{ass:Novikov}, every $\mathbb H$-local martingale $M$ admits the decomposition 
$$M_t = M_0 + \int_0^t \varphi_s dI_s + \int_0^t \psi_s dI^1_s + \int_0^t \int_{\mathbb R} w(s,q) m^{\pi}(ds, dq),$$
where $\varphi$ and $\psi$ are $\mathbb H$-adapted processes, and $w$ is an $\mathbb H$-predictable process indexed by $\mathbb R$ such that a.s.
$$\int_0^t \varphi^2_s ds < \infty, \quad \int_0^t \psi^2_s ds < \infty, \quad \int_0^t  \int_{\mathbb R} |w(s,q)| m^{p, \mathbb H}(dt, dq) < \infty, \quad t\geq 0.$$
\end{proposition}

We are now in the position to prove the following fundamental result, whose proof is postponed to Appendix \ref{app:proofs}.

\begin{theorem}
\label{filtering:eq}
\label{filteringeq}
Recall \eqref{mpi}, let ${\underline y} \in \mathcal{Y} $ be the initial distribution of $Z$, and let Assumptions \ref{ass:eta} and \ref{ass:Novikov} hold. Then the filter ${\underline \pi}_t := (\pi_t(i); i\in S)_{t \geq 0}$ solves the Kushner-Stratonovich system
\begin{eqnarray}
\label{KS}
\pi_t(i) &=& y_i + \int_0^t \sum_{j=1}^Q \lambda_{ji} \pi_s(j) ds + \int_0^t \pi_s(i) \sigma^{-1}  \Big\{ \beta_i -  \sum_{j=1}^Q  \beta_j \pi_s(j) \Big\} dI_s \nonumber \\
& & + \int_0^t  \pi_s(i)\Big\{ \alpha(\eta_s, i) -  \sum_{j=1}^Q  \alpha(\eta_s, j) \pi_s(j) \Big\} dI^1_s  + \int_0^t \int_{\mathbb R} \big( w^\pi_i(s,q) - \pi_{s^-}(i)\big) m^\pi(ds, dq),
\end{eqnarray}
for any $i\in S$. Here, $\beta_i = r -g(i)$ and 
\begin{equation}
\label{salto}
\displaystyle w^\pi_i(s,q):=\frac{ d \lambda^N(i) \pi_{s^-}(i) \mathds{1}_{\{ c(\eta_{s^-}, i) \neq 0\}} \delta_{c(\eta_{s^-}, i)}(dq)} {d\Big[ \sum_{j=1}^Q\pi_{s^-}(j) \lambda^N(j) \mathds{1}_{\{ c(\eta_{s^-}, j) \neq 0\}} \delta_{c(\eta_{s^-}, j)}(dq) \Big]}
\end{equation}
denotes the Radon-Nikodym derivative of the measure $\lambda^N(i)\pi_{s^-}(i) \mathds{1}_{\{ c(\eta_{s^-}, i) \neq 0\}} \delta_{c(\eta_{s^-}, i)}(dq) $ with respect to $\sum_{j=1}^Q\pi_{s^-}(j) \lambda^N(j) \mathds{1}_{\{ c(\eta_{s^-}, j) \neq 0\}} \delta_{c(\eta_{s^-}, j)}(dq)$.
\end{theorem}

Let us introduce the sequence of jump times and jump sizes of the process $\eta$, denoted by $\{T_n, \zeta_n \}_{n\geq 1}$, and recursively defined as $T_1 := \inf \{ t > 0:  \int_0^t  c(\eta_{s^-}, Z_{s^-} )dN_s \neq 0 \}$, 
$$T_{n+1} := \inf \big\{ t > T_{n}:  \int_{T_{n}}^t  c(\eta_{s^-}, Z_{s^-} )dN_s \neq 0 \big\}, \quad \zeta_n :=  \eta_{T_n} - \eta_{T_n-} =  c(\eta_{T_n^-}, Z_{T_n^-} ), \quad n\geq 1.$$
In the definitions above we use the standard convention that $\inf \emptyset = +\infty$.

Then the integer-valued measure associated to the jumps of $\eta$ (cf.\ \eqref{measure-m}) can also be written as 
\begin{equation} 
\label{measure-m1}
m(dt, dq) = \sum_{n \geq 1} \delta_{( T_n,  \zeta_n  )}(ds, dq) \mathds{1}_{\{ T_n < +\infty\}} .
\end{equation}
The filtering system of equations \eqref{KS} has a natural recursive structure in terms of the sequence $\{ T_n \}_{n\geq 1}$, as it is shown in the next proposition.

\begin{proposition}\label{recursive}
Between two consecutive jump times,  $t \in [T_n, T_{n+1})$,  the filtering system  of equations \eqref{KS} reads as 
\begin{eqnarray} \label{KSrec}
\pi_t(i) &=& \pi_{T_n}(i) + \int_{T_n}^t  \Big\{ \sum_{j=1}^Q \lambda_{ji} \pi_s(j) -  \pi_s(i) \Big[ \lambda^N(i) \mathds{1}_{\{ c(\eta_{s^-}, i) \neq 0\}}  - \sum_{j=1}^Q   \lambda^N(j)\pi_s(j) \mathds{1}_{\{ c(\eta_{s^-}, j) \neq 0\}} \Big]  \Big\} ds \nonumber \\
& & + \int_{T_n}^t \sigma^{-1}  \pi_s(i) \Big\{ \beta_i  -  \sum_{j=1}^Q  \beta_j \pi_s(j) \Big\} dI_s + \int_{T_n}^t  \pi_s(i)\Big\{ \alpha(\eta_s, i)  -  \sum_{j=1}^Q  \alpha(\eta_s, j) \pi_s(j) \Big\} dI^1_s,  \nonumber \\
\end{eqnarray}
for any $i\in S$. At a jump time of $\eta$, say $T_n$, ${\underline \pi}_t = (\pi_t(i); i\in S)_{t \geq 0}$ jumps as well,  and its value is given by
\begin{equation} 
\label{jump}\pi_{T_n}(i) = \frac{\lambda^N(i) \pi_{T_n^-}(i) \mathds{1}_{\{ \zeta_n = c(\eta_{T_n^-}, i) \}} } { \sum_{j=1}^Q\lambda^N(j) \pi_{T_n^-}(j) \mathds{1}_{\{ \zeta_n = c(\eta_{T_n^-}, j) \}} }, \quad  i \in S.
\end{equation}
\end{proposition}

\begin{proof}
First, recalling that $m^{\pi}(dt, dq)= m(dt, dq) - m^{p, \mathbb H}(dt, dq),$
and that
$$m^{p, \mathbb H}(dt, dq) = \sum_{j=1}^Q\pi_{t^-}(j) \lambda^N(j) \mathds{1}_{\{ c(\eta_{t^-}, j) \neq 0\}} \delta_{c(\eta_{t^-}, j)}(dq) dt,$$
we obtain that 
$$
\int_0^t \int_{\mathbb R} \big( w^\pi_i(s,q) - \pi_{s^-}(i)\big) m^{p, \mathbb H}(ds, dq) =  \int_0^t \pi_s(i) \Big[ \lambda^N(i) \mathds{1}_{\{ c(\eta_{s^-}, i) \neq 0\}}  - \sum_{j=1}^Q   \lambda^N(j)\pi_s(j) \mathds{1}_{\{ c(\eta_{s^-}, j) \neq 0\}} \Big]  ds,
$$ 
which, from \eqref{KS}, implies that for any $t \in [T_n, T_{n+1})$, $\pi_t(i)$ solves equation \eqref{KSrec}.

Finally, equation \eqref{jump} follows by \eqref{salto} and
$$
\pi_{T_n}(i) = w^\pi_i(T_n,\zeta_n)  = \frac{  \lambda^N(i) \pi_{T_n^-}(i) \mathds{1}_{\{ c(\eta_{T_n^-}, i) \neq 0\}} \delta_{c(\eta_{T_n^-}, i)}(\zeta_n)} { \sum_{j=1}^Q\pi_{T_n^-}(j) \lambda^N(j) \mathds{1}_{\{ c(\eta_{T_n^-}, j) \neq 0\}} \delta_{c(\eta_{T_n^-}, j)}(\zeta_n) }.
$$
\end{proof}

We want to stress that equation \eqref{jump} shows that the vector ${\underline \pi}_{T_n}$ is completely determined by the observed data $\eta$ and by the knowledge of ${\underline \pi}_{t}$ for $t \in [T_{n-1}, T_n)$, since $\pi_{T_n^-}(i) := \lim_{t  \uparrow T_n} \pi_{t}(i)$, $i \in S$. 

\begin{remark}
\label{rem:particularcases}
A few comments on the filtering equation are worth being done.
\begin{enumerate}
\item In the case $c(q,i) \equiv c \neq 0$, for any $i \in S$ and $q \in  \I$, the sequences of jump times of $\eta$ and $N$ coincide, and the filtering system of equations \eqref{KS} reduces to the simpler 
\begin{align*}
& \pi_t(i) = y_i+ \int_0^t \sum_{j=1}^Q \lambda_{ji} \pi_s(j) ds + \int_0^t \pi_s(i)\sigma^{-1}  \Big\{ \beta_i   -   \sum_{j=1}^Q  \beta_j \pi_s(j) \Big\} dI_s \nonumber \\
& + \int_0^t \pi_s(i)\Big\{ \alpha(\eta_s, i)  -  \sum_{j=1}^Q  \alpha(\eta_s, j) \pi_s(j) \Big\} dI^1_s  \nonumber \\
& + \int_0^t \left [ \frac{\lambda^N(i) \pi_{s^-}(i)}{\sum_{j=1}^Q\pi_{s^-}(j) \lambda^N(j) } - \pi_{s^-}(i) \right ] \Big( dN_s - \sum_{j=1}^Q\pi_{s^-}(j) \lambda^N(j) ds\Big), \quad i \in S.
\end{align*}

\item In the case $\alpha(q,i)=\alpha(i)$ and $c(q,i) \equiv 0$, for any $i \in S$ and $q \in  \I$, the filtering system of equations \eqref{KS} does not depend anymore explicitly on the process $\eta$. In particular, one has 
\begin{align}
\label{filteringetaDBM}
& \pi_t(i) = y_i + \int_0^t \sum_{j=1}^Q \lambda_{ji} \pi_s(j) ds + \int_0^t  \pi_s(i)\sigma^{-1}  \Big\{ \beta_i - \sum_{j=1}^Q  \beta_j \pi_s(j) \Big\} dI_s \nonumber \\
& + \int_0^t \pi_s(i)\Big\{ \alpha_i - \sum_{j=1}^Q  \alpha_j \pi_s(j) \Big\} dI^1_s, \quad i \in S,
\end{align}
where we have set $\alpha_i := \sigma_2^{-1}  \big\{b_1(i)  - \sigma^{-1} \beta_i \sigma_1\big\}$. 
With reference to \eqref{eta} and \eqref{alfa}, this setting corresponds, e.g., to the purely diffusive arithmetic case $c(q,i) = 0$,  $b_1(q,i) = b_1(i)$ and $\sigma_1(q) = \sigma_1 >0$, $\sigma_2(q) = \sigma_2 >0$, for any $i \in S$ and $q \in  \I$, or to the purely diffusive geometric case $c(q,i) = 0$,  $b_1(q,i) = b_1(i)q$ and $\sigma_1(q) = \sigma_1 q$, $\sigma_2(q) = \sigma_2 q$, for any $i \in S$ and $q \in  \I$. In Section \ref{sec:casestudy} we will provide the explicit solution to the optimal debt reduction problem within this setting.
\end{enumerate}
\end{remark}


\subsection{The Separated Problem}
\label{separated}

Thanks to the introduction of the filter, equations \eqref{Xzero}, \eqref{eta}, and \eqref{eq:dynX} can now be rewritten in terms of observable processes. In particular, we have that 
\begin{equation}
\label{Xzero1}
d X_t^{0} = \pi_t(\beta) X_t^{0} dt + \sigma X_t^{0} dI_t , \quad X^{0}_0=x>0,
\end{equation}
\begin{equation}
\label{eta1}
d \eta_t= \pi_t(b_1(\eta_t, \cdot)) dt + \sigma_1( \eta_t) dI_t + \sigma_2(\eta_t) dI^1_t + \int_\mathbb R \zeta  m(dt, d\zeta), \quad \eta_{0}=q \in  \I,
\end{equation}
and 
\begin{equation}
\label{eq:dynX1}
dX_t^{\nu} = \pi_t(\beta)  X_t^{\nu} dt + \sigma X_t^{\nu} dI_t - d \nu_t, \quad X^{\nu}_{0^-}=x >0. 
\end{equation}
\noindent Notice that, for any $\nu \in \mathcal M(x,\underline y,q)$, the process $X^{\nu}$ turns out to be $\mathbb H$-adapted, and depends on the vector ${\underline \pi}_t =(\pi_t(i); i\in S)_{t \geq 0}$, such that ${\underline \pi}_{0}= {\underline y} \in \mathcal{Y}$. 

\begin{definition}
\label{strong uniqueness}
(Strong Uniqueness). We say that a process $( { \underline {\widetilde  \pi}}_t, \widetilde \eta_t )_{t \geq 0}$ with values in $\mathcal{Y} \times \mathcal{I}$ is a strong solution to equations  \eqref{KS} and \eqref{eta1} if it satisfies pathwise those equations. We say that strong uniqueness for the system of equations  \eqref{KS} and \eqref{eta1} holds if, for any $( { \underline {\widetilde  \pi}}_t, \widetilde \eta_t )_{t \geq 0}$ strong solution to system \eqref{KS} and \eqref{eta1}, one has 
${\underline {\widetilde  \pi}}_t = {\underline \pi}_t$,  $ \widetilde \eta_t = \eta_t$, a.s.\ for all $t \geq 0$.
\end{definition}

\begin{proposition}
\label{uniq}
Let Assumptions \ref{ass:eta} and \ref{ass:Novikov} hold, and suppose that $\alpha(\cdot,i)$ is locally-Lipschitz for any $i\in S$, and there exists $M>0$ such that $|\alpha(q,i)|\leq M(1 + |q|)$, for any $q \in  \I$ and any $i\in S$. Then system \eqref{KS} and \eqref{eta1} admits a unique strong solution.
\end{proposition}

Notice that, under Assumption \ref{ass:eta}, the requirement on $\alpha$ of Proposition \ref{uniq} is verified, e.g., whenever $\sigma_2(q) \geq \kappa$, for some $\kappa>0$ and for any $q \in  \I$, or if $b_1/\sigma_2$ and $\sigma_1/\sigma_2$ are locally-Lipschitz on $q \in  \I$ and have sublinear growth.
The proof of Proposition \ref{uniq} is postponed to Appendix \ref{app:proofs}. As a byproduct, it also ensures strong uniqueness of the solution to \eqref{eq:dynX1}. In the following, when there will be the need to stress the dependence with respect to the initial value $x>0$, we shall denote the solution to \eqref{Xzero1} and \eqref{eq:dynX1} by $X^{x,0}$ and $X^{x,\nu}$, respectively. 

Since
$$\E\left[ \pi_t\big( h(X_t^{x,\nu}, \cdot) \big) \right] =  \E\left[\E\left[ h(X_t^{x,\nu}, Z_t ) |  \mathcal{H}_t \right] \right],$$
an application of Fubini-Tonelli's theorem allows to rewrite also the cost functional of \eqref{costfunctional} in terms of observable quantities as
\begin{equation}\label{c1}
\mathcal J_{x,\underline y,q}(\nu) = \E_{(x,\underline{y},q)}\left[\int_0^{\infty} e^{-\rho t} \pi_t( h(X_t^{\nu}, \cdot ) ) dt + \int_0^{\infty} e^{-\rho t} d \nu_t\right].
\end{equation}
Here $\E_{(x,\underline{y},q)}$ denotes the expectation conditioned on $X^{\nu}_{0^-}=x>0$, $\underline{\pi}_{0}=\underline{y} \in \mathcal{Y}$, and $\eta_{0}=q  \in  \I$. Notice that the latter expression does not depend anymore on the unobservable process $Z$, and this allows us to introduce a control problem with complete information, \emph{the separated problem}, in which the new state variable is given by the triplet $(X^{\nu}, \underline{\pi}, \eta)$. For this problem we rewrite the set $\mathcal{M}(x,{\underline y},q) $ in terms of the observable processes given by \eqref{KS}, \eqref{eta1} and \eqref{eq:dynX1}, and we denote by $\mathcal{A}(x,{\underline y},q) $ such a representation of the set $\mathcal{M}(x,{\underline y},q) $; that is,
\begin{eqnarray*}
\mathcal  A(x,{\underline y},q) &:=& \Big\{\nu:\Omega \times \mathbb R_+ \rightarrow \mathbb R_+ : {(\nu_t(\omega) := \nu (\omega,t))}_{t \ge 0} \ \textrm{is non decreasing, right-continuous,} \\
& &  \mathbb H\textrm{-adapted, such that} \ X_t^{x,\nu} \ge 0  \ \forall \ t \ge 0,\, \, X_{0^-}^{x,\nu} = x,\,\, \underline{\pi}_{0^-}= {\underline y}, \  \eta_{0}=q \ \textrm{a.s.}  \Big\},
\end{eqnarray*}
for every $x \in (0,\infty)$ initial value of $X^{x,\nu}$ defined in \eqref{eq:dynX1}, for any ${\underline y} \in \mathcal{Y}$ initial values of the process ${\underline \pi}_t = (\pi_t(i); i\in S)_{t \geq 0}$ solution to equation \eqref{KS}, and for any $q \in  \I$ initial value of $\eta$. In the following, we set $\nu_{0^-}=0$ a.s.\ for any $\nu \in \mathcal  A(x,{\underline y},q)$.

Given $\nu \in \mathcal A(x, {\underline y}, q)$, the triplet $\{ (X_t^{x,\nu}, {\underline \pi}_t, \eta_t) \}_{t \geq 0}$ solves \eqref{eq:dynX1}, \eqref{KS} and \eqref{eta1} and the jump measure associated to $\eta$ has $\mathbb H$-predictable dual projection given by equation \eqref{dual}. Hence, the process $\{ (X_t^{x,\nu}, {\underline \pi}_t, \eta_t) \}_{t \geq 0}$ is an $\mathbb H$-Markov process and we therefore define the Markovian separated problem as
\begin{equation*}
{\textbf{(P2)}}
\left\{
\begin{array}{l}
\displaystyle V(x, {\underline y}, q) := \inf_{\nu \in \mathcal A(x, {\underline y}, q)} \E_{(x,\underline{y}, q)}\bigg[\int_0^{\infty} e^{-\rho t} \pi_t( h(X_t^{\nu}, \cdot)) dt  + \int_0^{\infty} e^{-\rho t} d \nu_t \bigg] \quad \textrm{with} \\
\\
dX_t^{x,\nu} = \pi_t(\beta)  X_t^{x,\nu} dt + \sigma X_t^{x,\nu} dI_t - d \nu_t, \quad X^{x,\nu}_{0-}=x >0,  \\   \\ 
({\underline \pi}, \eta) \ \textrm{ solution to equations  \eqref{KS} and \eqref{eta1}.}
\end{array}
\right.
\end{equation*}
This is now a \emph{singular stochastic problem under complete information}, since all the processes involved are $\mathbb H$-adapted. 

The next proposition immediately follows from the previous construction of the separated problem, and from the strong uniqueness of the solutions to \eqref{KS}, \eqref{eta1}, and \eqref{eq:dynX1}.
\begin{proposition}
\label{equivalence}
Assume strong uniqueness for the system of equations \eqref{KS} and \eqref{eta1}, and let $(x,\underline{y},q) \in (0,\infty) \times \mathcal{Y} \times \I$ be the initial values of the process $(X,Z,\eta)$ in the problem under partial observation $\textbf{(P1)}$. Then 
\begin{equation}
\label{eq:equivalence}
V_{po}(x,\underline{y},q)=V(x,\underline{y},q).
\end{equation}

Moreover, $\nu^* \in \mathcal A(x,\underline{y},q)$ is an optimal control for the separated problem $\textbf{(P2)}$ if and only if $\nu^* \in \mathcal M(x,\underline{y},q)$ is an optimal control for the original problem under partial observation $\textbf{(P1)}$.
\end{proposition}

\begin{remark}
\label{rem:no-eta}
Notice that in the setting of Remark \ref{rem:particularcases}-(2), the pair $(X^{x,\nu}, {\underline \pi})$ solving equations \eqref{eq:dynX1} and \eqref{KS}, respectively, is an $\mathbb H$-Markov process, for any $\nu \in \mathcal A(x,\underline{y},q)$, $(x,\underline{y},q) \in (0,\infty) \times \mathcal{Y} \times \I$. As a consequence, since the cost functional and the set of admissible controls do not depend explicitly on the process $\eta$, the value function of the separated problem $\textbf{(P2)}$ does not depend anymore on the variable $q$. We will consider this setting as a case study in Section \ref{sec:casestudy}. 
\end{remark}


\subsection{A Probabilistic Verification Theorem via Reduction to Optimal Stopping}
\label{sec:verificationthm}

In this section we relate the separated problem to a Markovian optimal stopping problem, and we show that the solution to the latter is directly related to the optimal control of the former. The following analysis is fully probabilistic and it is based on a change of variable formula for Lebesgue-Stieltjes integrals that has been already employed in singular control problems (see, e.g., \cite{BK} and \cite{Ferrari18}). The result of this section will then be employed in Section \ref{sec:casestudy} where, in a case study, we determine the expression of the optimal debt reduction policy by solving an auxiliary optimal stopping problem.

With regard to Problem \textbf{(P2)}, notice that $\pi_t\big(h(X_t^{x,\nu}, \cdot) \big) = \sum_{i=1}^Q \pi_t(i) h(X_t^{x,\nu}, i)$ a.s.\ for any $t\geq0$. Then, for any $(x,\underline{\pi}) \in (0,\infty) \times \mathcal{Y}$, set
\begin{equation}\label{hat_h}
\widehat{h}(x, \underline{\pi}):=\sum_{i=1}^Q \pi(i) h(x,i),
\end{equation}
and, given $z \in (0,\infty)$, we introduce the optimal stopping problem 
\begin{equation}
\label{valueOS-general}
\widetilde{U}_t(z):=\essinf_{\tau \geq t}\E\bigg[\int_t^{\tau} e^{-\rho (s-t)} X^{1,0}_s\, \widehat{h}_x(X^{z,0}_s, \underline{\pi}_s) ds + e^{-\rho (\tau - t)} X^{1,0}_\tau \,\Big|\,\mathcal{H}_t\bigg], \qquad t\geq 0,
\end{equation}
where the optimization is taken over all the $\mathbb{H}$-stopping times $\tau \geq t$. 

Under Assumption \ref{ass:h}, the expectation in \eqref{valueOS-general} is finite for any $\mathbb{H}$-stopping time $\tau \geq t$, for any $t\geq0$. To take care of the event $\{\tau = \infty\}$, in \eqref{valueOS-general} we make use of the convention 
\begin{equation}
\label{eq:convention}
e^{-\rho \tau} X^{1,0}_{\tau}:=\liminf_{t\uparrow \infty}e^{-\rho t} X^{1,0}_t \quad \text{on}\quad \{\tau = \infty\}.
\end{equation}

Denote by $U_t(z)$ a c\`adl\`ag modification of $\widetilde{U}_t(z)$, and  observe that  $0 \leq U_t(z) \leq X^{1,0}_t $, for any $t \geq 0$, a.s. Also, define the stopping time 
\begin{equation}
\label{taustargeneral}
\tau_t^*(z):=\inf\{s\geq t: U_s(z) \geq X^{1,0}_s \}, \quad z \in (0,\infty),
\end{equation}
with the convention that $\tau_t^*(z) = \infty$ if the set on the right-hand side is empty. Then by Theorem D.12 in Appendix D of \cite{KS2}, $\tau_t^*(z)$ is an optimal stopping time for problem \eqref{valueOS-general}. In particular, $\tau^*(z):= \tau_0^*(z)$ is optimal for the problem
\begin{equation}
\label{valueOS-general1}
U_0(z):=\inf_{\tau \geq 0}\E \left[ \int_0^{\tau} e^{-\rho t}  X^{1,0}_t\, \widehat{h}_x(X^{z,0}_t, \underline{\pi}_t) dt + e^{-\rho \tau} X^{1,0}_\tau \right]
\end{equation}

Notice that since $h_x(\cdot, \underline{\pi})$ is a.s.\ increasing, then $z \mapsto \tau^*(z)$ is a.s.\ decreasing. Such monotonicity of $\tau^*(\,\cdot\,)$ will be important in the following as we will need to consider its generalized inverse. Moreover, since  the triplet $(X^{z,0}_t, \underline{\pi}_t, \eta_t)$ is an homogenous $\mathbb{H}$-Markov process, there exists a measurable function $U: (0,\infty) \times \mathcal{Y} \times \I \to \mathbb{R}$ such that $U_t(z) = U (X^{z,0}_t, \underline{\pi}_t, \eta_t)$ for any $t \geq 0$, a.s. Hence, $U_0(z)=U(z,\underline{y},q)$, and for any $(x,\underline{y},q) \in (0,\infty) \times \mathcal{Y} \times \I$, define
\begin{equation}
\label{hatV}
\widetilde{V}(x, \underline{y},q):= \int_{0}^{x} U(z, \underline{y},q)dz.
\end{equation}

Moreover, introduce the nondecreasing, right-continuous process 
\begin{equation}
\label{candidate-general}
\overline{\nu}^*_t := \sup\{\alpha \in [0,x]: \tau^*(x-\alpha) \leq t\}, \quad t \geq 0, \qquad \overline{\nu}^*_{0^-} =0,
\end{equation}
and then also the process 
\begin{equation}
\label{nustar-general}
\nu^*_t:=\int_{0}^t X^{1,0}_s d\overline{\nu}^*_s, \quad t > 0, \qquad \nu^*_{0^-}=0. 
\end{equation}
\noindent Notice that $\overline{\nu}^*_{\cdot}$ is the right-continuous inverse of $\tau^*(\,\cdot\,)$.

\begin{theorem}
\label{teo:verifico}
Let $\widetilde{V}$ be as in \eqref{hatV} and $V$ as in the definition of Problem \textbf{(P2)}. Then one has $\widetilde{V} = V$, and $\nu^*$ is the (unique) optimal control for Problem \textbf{(P2)}.
\end{theorem}

\begin{proof}
\emph{Step 1.} Let $x > 0$, $\underline{y} \in \mathcal{Y}$, and $q \in  \I$ be given and fixed. For $\nu \in \mathcal{A}(x, \underline y, q)$, we introduce the process $\overline{\nu}$ such that
$\overline{\nu}_t:= \int_{0}^t \frac{d{\nu}_s}{X^{1,0}_s}$, $t\geq 0$, and define its inverse (see, e.g., Chapter 0, Section 4 of \cite{RY}) by
\begin{equation}
\label{taunu}
\tau^{\overline{\nu}}(z):=\inf\{t\geq 0 \ | \ x - \overline{\nu}_t < z\}, \qquad 0 < z \leq x.
\end{equation}
Notice that the process $\tau^{\overline{\nu}}(z):=\{\tau^{\overline{\nu}}(z),\  z \leq x\}$ has decreasing, left-continuous sample paths, and hence it admits right-limits
\begin{equation}
\label{taunumeno}
\tau^{\overline{\nu}}_{+}(z):=\inf\{t\geq 0 \ | \ x - \overline{\nu}_t \leq z\}, \qquad z \leq x.
\end{equation}
Moreover, the set of points $z\in\mathbb{R}$ at which $\tau^{\overline{\nu}}(z)(\omega) \neq \tau^{\overline{\nu}}_{+}(z)(\omega)$ is a.s.\ countable for a.e.~$\omega\in\Omega$.

The random time $\tau^{\overline{\nu}}(z)$ is actually an $(\mathcal{H}_{t})$-stopping time because it is the entry time of an open set of the right-continuous process $\overline{\nu}$, and $(\mathcal{H}_t)_{t\ge0}$ is right-continuous. Moreover, since $\tau^{\overline{\nu}}_{+}(z)$ is the first entry time of the right-continuous process $\overline{\nu}$ into a closed set, it is an $(\mathcal{H}_{t})$-stopping time as well for any $z \leq x$.

Proceeding then as in Step 1 of the proof of Theorem 3.1 in \cite{Ferrari18}, by employing the change of variable formula in Chapter 0, Proposition 4.9 of \cite{RY}, one finds that
$$\widetilde{V}(x, \underline{y}, q) = \int_{0}^{x}U(z, \underline{y}, q)dz \leq \mathcal{J}_{x,\underline{y}, q}(\nu).$$
Hence, since $\nu$ was arbitrary, we find
\begin{equation}
\label{ineq1}
\widetilde{V}(x, \underline{y}, q) \leq V(x, \underline{y}, q), \quad (x,\underline{y},q) \in (0,\infty) \times \mathcal{Y} \times \I.
\end{equation}
\vspace{0.1cm}

\emph{Step 2.} To complete the proof we have to show the reverse inequality. Let $x \in (0,\infty)$, $\underline{y} \in \mathcal{Y}$, and $q \in  \I$, initial values of $X^{x,\nu}$,  $\underline{\pi}$ and $\eta$. We first notice that $\nu^* \in \mathcal{A}(x, \underline{y},q)$. Indeed, $\nu^*$ is nondecreasing, right-continuous and such that $X^{x,\nu^*}_t  = X^{1,0}_t(x - \overline{\nu}^*_t) \geq 0$ a.s.\ for all $t\geq 0$, since one has by definition $\overline{\nu}^*_t \leq x$ a.s. Moreover, for any $0<z \leq x$, we can write (cf.\ \eqref{candidate-general} and \eqref{taunumeno})
$$\tau^{\overline{\nu}^*}_+(z) \leq t \,\, \Longleftrightarrow \,\, \overline{\nu}^*_t \geq x - z\,\, \Longleftrightarrow \,\, \tau^*(z) \leq t.$$

Then, recalling that $\tau^{\overline{\nu}^{*}}_{+}(z)=\tau^{\overline{\nu}^{*}}(z)$ $\mathbb{P}$-a.s.~and for a.e.~$z\le x$, we pick $\nu=\nu^*$ (equivalently, $\overline{\nu}=\overline{\nu}^*$), and following Step 2 in the proof of Theorem 3.1 of \cite{Ferrari18}, we obtain $\widetilde{V}(x, \underline{y}, q)=\mathcal{J}_{x,\underline{y}, q}(\nu^*)$. That is, $\widetilde{V}=V$ by \eqref{ineq1} and admissibility of $\nu^*$. Therefore $\nu^*$ is optimal. In fact, $\nu^*$ is the unique optimal control in the class of controls belonging to $\mathcal{A}(x, \underline{y}, q)$ and such that $\mathcal{J}_{x,\underline{y}, q}(\nu) < \infty$ by strict convexity of $\mathcal{J}_{x,\underline{y}, q}(\,\cdot\,)$.
\end{proof}

\begin{remark}
\label{rem:OS-Markov-case}
For any given $(x,\underline{y},q) \in (0,\infty) \times \mathcal{Y} \times \I$, define the Markovian optimal stopping problem
\begin{equation*}
v(x, \underline{y}, q):= \inf_{\tau \geq 0}\E_{(x, \underline{y} , q)}\bigg[\int_0^{\tau} e^{-\rho t}   X^{x,0}_t \widehat{h}_x(  {X^{x,0}_t}, \underline{\pi}_t) dt + e^{-\rho \tau} X^{x,0}_\tau \bigg],
\end{equation*}
where $\E_{(x, \underline{y}, q)}$ denotes the expectation under the probability measure $\P_{(x,\underline{y},q)}$ such that $\P(\,\cdot\,):=\P(\,\cdot\,| X^{x,0}_0=x, \underline{\pi}_0=\underline{y}, \eta_0 = q)$. Then, it is readily verified that $v(x, \underline{y}, q)= x U(x, \underline{y}, q)$. Moreover, it holds that the stopping time
$$\tau^*(x):=\inf\{t\geq 0: v(X^{x,0}_t, \underline{\pi}_t , \eta_t) \geq X^{x,0}_t \}, \quad \P_{(x,\underline{y},q)} - \textrm{a.s}$$
is optimal for $v(x,\underline{y}, q)$.
\end{remark}


\section{The Solution in a Case Study with $Q=2$ Economic Regimes}
\label{sec:casestudy}

In this section, we build on the general filtering analysis developed in the previous sections and on the result of Theorem \ref{teo:verifico}, and we provide the form of the optimal debt reduction policy in a case study that is defined through the following standing assumption.

\begin{assumption}
\label{ass:casestudy}
\begin{enumerate}\hspace{10cm}
\item $Z$ takes values in $S=\{1, 2\}$, and, with reference to \eqref{eq:dynX}, we take $g_2:=g(2)<g(1)=:g_1$;
\item for any $q\in \I$ and any $i\in \{1, 2\}$ one has $c(q,i) = 0$ and, for $\alpha$ as in \eqref{alfa}, we take $\alpha(q,i)=\alpha(i)$;
\item $h(x,i)=h(x)$ for all $(x,i) \in (0,\infty) \times \{1,2\}$, with $h:\mathbb{R} \to \mathbb{R}$ such that: 
\begin{itemize}
\item[(i)] $x \mapsto h(x)$ is strictly convex, twice-continuously differentiable, and nondecreasing on $\mathbb{R}_+$ with $h(0)=0$ and $\lim_{x \uparrow \infty}h(x)=\infty$; 
\item[(ii)] there exist $\gamma > 1$, $0<K_o<K$ and $K_1,K_2>0$ such that 
$$K_o|x^+|^{\gamma} - K \leq h(x) \leq K(1 + |x|^{\gamma}),\quad |h'(x)| \leq K_1(1 + |x|^{\gamma-1})$$
and
$$|h''(x)| \leq K_2(1 + |x|^{(\gamma-2)^+}).$$
\end{itemize}
\end{enumerate}
\end{assumption}

Notice that under Assumption \ref{ass:casestudy}-(2) the macroeconomic indicator $\eta$ has a suitable diffusive dynamics whose coefficients $b_1,\sigma_1,\sigma_2$ are such that the function $\alpha$ is independent of $q$. As discussed in Remark \ref{rem:particularcases}-$(2)$, this is the case of a geometric or arithmetic diffusive dynamics for $\eta$. In this setting the Kushner-Stratonovich system \eqref{KS} reduces to
\begin{equation}
\label{KS1}
d\pi_t(1) = \big[\lambda_2 - (\lambda_1 + \lambda_2) \pi_t(1) \big] dt + \pi_t(1)(1-  \pi_t(1) )\Big[\frac { \beta_1 - \beta_2} {\sigma} dI_t + (\alpha_1-\alpha_2)dI^1_t\Big], 
\end{equation}
and $\pi_t(2) = 1- \pi_t(1)$. Here, $\lambda_1:= \lambda_{12} >0$ and $\lambda_2 := \lambda_{21} >0$. 

Denoting by $\pi_t := \pi_t(1)$, $t\geq 0$, problem {\textbf{(P2)}} then reads as 
\begin{equation*}
{\textbf{(P3)}}
\left\{
\begin{array}{l}
\displaystyle V(x,y) = \inf_{\nu \in \mathcal A(x,y)}\E_{(x,y)}\bigg[\int_0^{\infty} e^{-\rho t} h(X_t^{\nu}) dt  + \int_0^{\infty} e^{-\rho t} d \nu_t\bigg] \quad \textrm{with} \\
 \\
\displaystyle dX_t^{x,y,\nu} = [\beta_2 + \pi^y_t(g_2 -g_1)] X_t^{x,y,\nu} dt + \sigma X_t^{x,y,\nu} dI_t - d \nu_t, \quad X^{x,y,\nu}_{0-}=x >0,  \\ \\

\displaystyle d\pi^y_t= \big[\lambda_2 - (\lambda_1+\lambda_2)\pi^y_t\big]dt + \pi^y_t(1-\pi^y_t)\Big[\frac{(g_2-g_1)}{\sigma}dI_t + (\alpha_1-\alpha_2)dI^1_t\Big], \quad \pi_0 =y \in (0,1),\\
\end{array}
\right.
\end{equation*}
where $g_i= r -\beta_i$, denotes the rate of economic growth in the state $i$, $i=1,2$.

It is worth noticing that there is no need to involve the process $\eta$ in the Markovian formulation of problem $\textbf{(P3)}$. This is due to the fact that the couple $(X^{\nu},\pi)$, solving the two stochastic differential equations above is a strong Markov process, and the cost functional and the set of admissible controls (denoted by $\mathcal A(x,y)$ above) do not depend explicitly on $\eta$. For this reason the value function of Problem {\textbf{(P3)}} does not depend on the initial value $q$ of the process $\eta$. However, memory of the macroeconomic indicator process $\eta$ appears in the filter $\pi$ through the constant term $\alpha_1 - \alpha_2$ in its dynamics.

Finally, we recall that, thanks to Propostion \ref{equivalence}, by solving Problem {\textbf{(P3)}} we are also solving the original problem {\textbf{(P1)}}. Indeed, we have that 
$$V_{po}(x,y) = V(x,y), \quad \text{for any given and fixed} \quad (x,y) \in (0,\infty) \times (0,1),$$ 
and a control is optimal for the separated problem {\textbf{(P3)}} if and only if it is such for the original problem under partial observation.

In the following analysis, we need (for technical reasons due to the infinite time-horizon of our problem) to take a discount factor sufficiently large. Namely, defining
\begin{align*}
\rho_o:= &  \big(\beta_2 + \frac{1}{2}\sigma^2\big) \vee \left[ \gamma\beta_2 + \frac{1}{2}\sigma^2\gamma(\gamma-1) \right] \vee \big(2\beta_2 + \sigma^2\big) \vee \left[ 24\theta^2 -(\lambda_1+\lambda_2) \right] \nonumber \\
& \vee \big(4\beta_2 + 6\sigma^2\big) \vee \left[ 4 \beta_2 (2\vee \gamma)  + 2\sigma^2(2\vee \gamma) \left( 4 (2\vee \gamma)-1 \right) \right],
\end{align*}
with $\theta^2:=\frac{1}{2}\big[\frac{(g_1-g_2)^2}{\sigma^2} + (\alpha_1-\alpha_2)^2\big]$, we assume the following.
\begin{assumption}
\label{ass:rho}
One has $\rho> \rho_o^+$.
\end{assumption}

Due to the growth condition on $h$, Assumption \ref{ass:rho} in particular ensures that $\rho > \gamma\beta_2 + \frac{1}{2}\sigma^2\gamma(\gamma-1)$ so that the (trivial) admissible control $\nu\equiv 0$ has a finite total expected cost.


\subsection{The Related Optimal Stopping Problem}
\label{sec:relatedOS}

Motivated by the results of the previous sections (see in particular Theorem \ref{teo:verifico}), we now aim at solving Problem {\textbf{(P3)}} through the study of an auxiliary optimal stopping problem. Informally, the solution to such an optimal stopping problem gives the optimal time at which the government should reduce the debt ratio by one additional unit. The optimal stopping problem involves a two-dimensional diffusive process, and in the following we provide an almost exclusively probabilistic analysis.

\subsubsection{Formulation and Preliminary Results}
\label{sec:diffusionOS}

Recall that $(I_t,I^1_t)_{t\geq 0}$ is a two-dimensional, standard $\mathbb{H}$-Brownian motion, and introduce the two-dimensional diffusion process $(\X,\pi):=(\X_t,\pi_t)_{t\geq0}$ solving the stochastic differential equations (SDEs)
\begin{equation}
\label{syst:XYhat}
\left\{
\begin{array}{lr}
\displaystyle d\X_t= \X_t \left[ \beta_2 +(g_2-g_1)\pi_t \right] dt + \sigma \X_t dI_t, \\[+8pt]
\displaystyle d\pi_t=\big[\lambda_2 - (\lambda_1+\lambda_2)\pi_t\big]dt + \pi_t(1-\pi_t)\Big[\frac{(g_2-g_1)}{\sigma}dI_t + (\alpha_1-\alpha_2)dI^1_t\Big],
\end{array}
\right.
\end{equation}
with initial conditions $\X_0=x$, $\pi_0=y$ for any $(x,y):=(0,\infty) \times (0,1)$. In the following, we set $\mathcal{O}:=(0,\infty) \times (0,1)$. Recall that $\beta_2= r - g_2$. 

Since the process $\pi$ is bounded, classical results on SDEs ensure that system \eqref{syst:XYhat} admits a unique strong solution, that, when needed, we shall denote by $(\X^{x,y},\pi^y)$ in order to stress its dependence on the initial datum $(x,y) \in \mathcal{O}$. In particular, one easily obtains 
\begin{equation}
\label{sol:X}
\X^{x,y}_t = x e^{(\beta_2 - \frac{1}{2}\sigma^2)t + \sigma I_t + (g_2-g_1)\int_0^t \pi^y_s ds}, \quad t \geq 0,
\end{equation} 

Moreover, it can be shown that the Feller's test of explosion (see, e.g., Chapter 5.5 in \cite{KS}) gives that $1= \P(\pi^y_t \in (0,1),\,\,\forall t\geq0)$ for all $y \in (0,1)$. In fact, the boundary points $0$ and $1$ are entrance-not-exit (cf.\ \cite{BoroSal}, p.\ 15), hence unattainable for the process $\pi$.

With regard to Remark \ref{rem:OS-Markov-case}, here we study the fully two-dimensional Markovian optimal stopping problem with value function
\begin{align}
\label{valueOS}
v(x,y) & := \inf_{\tau \geq 0}\E_{(x,y)}\bigg[\int_0^{\tau} e^{- \rho t} \X_t h'(\X_t) dt + e^{-\rho\tau} \X_{\tau}\bigg] =: \inf_{\tau \geq 0} \widehat{\mathcal{J}}_{(x,y)}(\tau), \qquad (x,y) \in \mathcal{O}.
\end{align}
In \eqref{valueOS} the optimization is taken over all the $\mathbb{H}$-stopping times, and the symbol $\E_{(x,y)}$ denotes the expectation under the probability measure $\P_{(x,y)}$ on $(\Omega,\mathcal{F})$, defined as $\P_{(x,y)}(\,\cdot\,):=\P(\,\cdot\,|\X_0=x,\pi_0=y)$, for any $(x,y)\in\mathcal{O}$.

Due to the fact that $\pi$ is positive, $g_2 - g_1 <0$, and $\rho>\beta_2$ by Assumption \ref{ass:rho}, one has from \eqref{sol:X} that
\begin{equation}
\label{conventionOS}
\liminf_{t \uparrow \infty} e^{-\rho t} \X_t =0 \quad \P_{(x,y)}-a.s.,
\end{equation}
which implies the convention (cf.\ \eqref{eq:convention}) $e^{-\rho \tau} \X_{\tau}=0$ on $\{\tau=\infty\}$.

Clearly, one has $v\geq 0$ since $\X$ is positive and $h$ is increasing on $\mathbb{R}_+$. Also, $v \leq x$ on $\mathcal{O}$, and we can therefore define the \emph{continuation region} and the \emph{stopping region} as
\begin{equation}
\label{continuation-stopping}
\mathcal{C}:=\{(x,y) \in \mathcal{O}:\, v(x,y) < x\}, \qquad \mathcal{S}:=\{(x,y) \in \mathcal{O}:\, v(x,y) = x\}.
\end{equation}

Notice that integrating by parts the term $e^{-\rho \tau} \X_{\tau}$, taking expectations, and exploiting that for any $\mathbb{H}$-stopping time $\tau$ one has $\E[\int_0^{\tau} e^{-\rho s} \X_s dI_s]=0$ (because $\rho>\beta_2 + \frac{1}{2}\sigma^2$ by Assumption \ref{ass:rho}), we can equivalently rewrite \eqref{valueOS} as
\begin{align}
\label{valueOS-bis}
v(x,y) & := x + \inf_{\tau \geq 0} \E_{(x,y)}\bigg[\int_0^{\tau} e^{-\rho t } \X_t \Big(h'(\X_t)- (\rho - \beta_2 -(g_2-g_1)\pi_t)\Big) dt\bigg],
\end{align}
for any $(x,y) \in \mathcal{O}$. 
From \eqref{valueOS-bis} it is readily seen that
\begin{equation}
\label{inclusion0}
\{(x,y)\in \mathcal{O}:\, h'(x) - (\rho - \beta_2 -(g_2-g_1)y) < 0 \} \subseteq \mathcal{C},
\end{equation}
which implies
\begin{equation}
\label{inclusion}
\mathcal{S} \subseteq \{(x,y)\in \mathcal{O}:\, h'(x) - (\rho - \beta_2 -(g_2-g_1)y) \geq 0 \},
\end{equation}
Moreover, since $\rho$ satisfies Assumption \ref{ass:rho}, and $0 \leq \pi_t \leq 1$ for any $(x,y) \in \mathcal{O}$, one has that
\begin{equation}
\label{integrability}
\E_{(x,y)}\bigg[\int_0^{\infty} e^{- \rho t} \X_t \Big(h'\big(\X_t\big) + \rho + |\beta_2| + |g_2-g_1| dt\Big)\bigg] < \infty,
\end{equation}
and the family of random variables 
$$
\bigg\{\int_0^{\tau} e^{-\rho t } \X_t \Big(h'(\X_t)- (\rho - \beta_2 -(g_2-g_1)\pi_t)\Big) dt,\,\,\tau\,\,\mathbb{H}-\text{stopping time}\bigg\}
$$ 
is therefore $\mathbb{H}$-uniformly integrable under $\P_{(x,y)}$. 

Preliminary properties of $v$ are given in the next proposition.
\begin{proposition}
\label{prop:continuousOS}
The following hold:
\begin{itemize}
\item[(i)] $x \mapsto v(x,y)$ is increasing for any $y \in (0,1)$;
\item[(ii)] $y \mapsto v(x,y)$ is decreasing for any $x \in (0,\infty)$;
\item[(iii)] $(x,y) \mapsto v(x,y)$ is continuous in $\mathcal{O}$.
\end{itemize}
\end{proposition}
\begin{proof}
We prove each claim separately.
\vspace{0.15cm}

\emph{(i)}. Recall \eqref{valueOS}. By the strict convexity and the monotonicity of $h$ and \eqref{sol:X}, it follows that $x \mapsto \widehat{\mathcal{J}}_{(x,y)}(\tau)$ is increasing for any $\mathbb{H}$-stopping time $\tau$, and for any $y \in (0,1)$. Hence the claim is proved.
\vspace{0.1cm}

\emph{(ii)}. This is due to the fact that $y \mapsto \widehat{\mathcal{J}}_{(x,y)}(\tau)$ is decreasing for any stopping time $\tau$ and any $x \in (0,\infty)$. Indeed, the mapping $y \mapsto \X^{x,y}_t$ is a.s.\ decreasing for any $t\geq0$ (because $y \mapsto \pi^y_t$ is a.s.\ increasing by the comparison theorem of Yamada and Watanabe - see, e.g., Proposition 2.18 in Chapter 5.2 of \cite{KS} - and $g_2-g_1<0$), and $x \mapsto xh'(x)$ is increasing. 
\vspace{0.1cm}

\emph{(iii)}. Since $(x,y) \mapsto (\X^{x,y}_t, \pi^y_t)$ is a.s.\ continuous for any $t\geq0$, it is not hard to verify that $(x,y) \mapsto \widehat{\mathcal{J}}_{(x,y)}(\tau)$ is continuous for any given $\tau\geq 0$. Hence, $v$ is upper semicontinuous. We now show that it is also lower semicontinuous.

Let $(x,y) \in \mathcal{O}$ and let $(x_n,y_n)_n \subseteq \mathcal{O}$ be any sequence converging to $(x,y)$. Without loss of generality, we may take $(x_n,y_n) \in (x-\delta,x+\delta) \times (y-\delta,y+\delta)$, for a suitable $\delta>0$. Letting $\tau^n_{\varepsilon}:=\tau^n_{\varepsilon}(x_n,y_n)$ be an $\varepsilon$-optimal for $v(x_n,y_n)$, but suboptimal for $v(x,y)$, we can then write
\begin{align}
\label{lsc}
& v(x,y) - v(x_n,y_n) \leq \E\bigg[\int_0^{\tau^n_{\varepsilon}} e^{-\rho t}\Big( \X^{x,y}_t h'\big(\X^{x,y}_t\big) - \X^{x_n,y_n}_t h'\big(\X^{x_n,y_n}_t\big)\Big) dt \bigg]  \\
&  + \E\Big[e^{-\rho \tau^n_{\varepsilon}} \Big(\X^{x,y}_{\tau^n_{\varepsilon}} - \X^{x_n,y_n}_{\tau^n_{\varepsilon}}\Big)\Big] + \varepsilon. \nonumber
\end{align}
Notice now that a.s.
\begin{align}
\label{domconv1}
& \int_0^{\tau^n_{\varepsilon}} e^{-\rho t}\Big| \X^{x,y}_t h'\big(\X^{x,y}_t\big) - \X^{x_n,y_n} h'\big(\X^{x_n,y_n}_t\big)\Big| dt \leq \int_0^{\infty} e^{-\rho t}\Big(\X^{x,y}_t h'\big(\X^{x,y}_t\big) + \X^{x+\delta,y-\delta}_t h'\big(\X^{x+\delta,y-\delta}_t\big)\Big) dt, \nonumber
\end{align}
where we have used that $x \mapsto \X^{x,y}$ is increasing, $y \mapsto \X^{x,y}$ is decreasing, and $x \mapsto xh'(x)$ is positive and increasing. The random variable on the right-hand side of the latter equation is independent of $n$ and integrable due to \eqref{integrability}.

Also, by an integration by parts, and performing standard estimates, we can write that a.s.
\begin{eqnarray}
\label{domconv2}
&& e^{-\rho \tau^n_{\varepsilon}} \Big(\X^{x,y}_{\tau^n_{\varepsilon}} - \X^{x_n,y_n}_{\tau^n_{\varepsilon}}\Big) 
\leq |x-x_n| + \int_0^{\infty}e^{-\rho s} \big(\rho + |\beta_2| + |g_2-g_1|\big)\big(\X^{x,y}_s + \X^{x+\delta,y-\delta}_s\big) ds, \nonumber
\end{eqnarray}
and the last integral above is independent of $n$ and it has finite expectation due to \eqref{integrability}.

Then, taking limits as $n\uparrow \infty$, invoking the dominated convergence theorem thanks to the previous estimates, and using that $(x,y) \mapsto (\X^{x,y}_t, \pi^y_t)$ is a.s.\
continuous for any $t\geq0$ we find (after rearranging terms) that
$$\liminf_{n\uparrow \infty}v(x_n,y_n) \geq v(x,y) - \varepsilon.$$
We thus conclude that $v$ is lower semicontinuous at $(x,y)$ by arbitrariness of $\varepsilon$. Since $(x,y) \in \mathcal{O}$ was arbitrary as well, then $v$ is lower semicontinuous on $\mathcal{O}$.
\end{proof}

Due to Proposition \ref{prop:continuousOS}-(iii) one has that the stopping region is closed, whereas the continuation region is open. Moreover, thanks to \eqref{integrability} and the $\P_{(x,y)}$-a.s.\ continuity of $t \mapsto \int_0^{t} e^{- \rho s} \X_s (h'(\X_s) - (\rho - \beta_2 -(g_2-g_1)\pi_s) ds$, we can apply Theorem D.12 in Appendix D of \cite{KS2} to obtain that the first entry time of $(\X,\pi)$ into $\mathcal{S}$ is optimal for \eqref{valueOS}; that is,
\begin{equation}
\label{OStime}
\tau^{\star}(x,y):=\inf\big\{t\geq 0:\,(\X_t, \pi_t) \in \mathcal{S}\big\}, \qquad \P_{(x,y)}-a.s., \qquad (x,y) \in \mathcal{O},
\end{equation}
attains the infimum in \eqref{valueOS} (here we adopt the usual convention $\inf \emptyset = \infty$). 

Also, by employing standard means based on the strong Markov property of $(\X,\pi)$ (see, e.g., \cite{PS}, Ch.\ I, Sec.\ 2, Thm.\ 2.4), one can show that, $\P_{(x,y)}$-a.s., the process $S:=\big(S_t\big)_{t\geq0}$, with 
\begin{align*}
\label{subharm1}
S_t:=\Big(e^{-\rho t}v(\X_{t},\pi_t) + \int_0^t e^{- \rho s} \X_s h'\big(\X_t\big) dt\Big)_{t\geq 0},\quad \text{is an $\mathbb{H}$-submartingale}, 
\end{align*}
and that the stopped process $(S_{t\wedge \tau^{\star}}\big)_{t\geq 0}$ is an $\mathbb{H}$-martingale.
The latter two conditions are usually referred to as the \emph{subharmonic characterization} of the value function $v$.

We now rule out the possibility of an empty stopping region.
\begin{lemma}
\label{lem:Snotempty}
The stopping region of \eqref{continuation-stopping} is not empty.
\end{lemma}
\begin{proof}
We argue by contradiction and we suppose that $\mathcal{S} = \emptyset$. Hence, for any $(x,y) \in \mathcal{O}$ we can write 
\begin{eqnarray}
\label{contra}
&& x > v(x,y) =   \E_{(x,y)}\bigg[\int_0^{\infty} e^{-\rho t} \X_t h'(\X_t) dt\bigg] \geq K_o\, x^{\gamma}\, \E_{(1,y)}\bigg[\int_0^{\infty} e^{-\rho t} \big(\X_t\big)^{\gamma} dt\bigg] - \frac{K}{\rho},
\end{eqnarray}
where the inequality $xh'(x) \geq h(x)$, due to convexity of $h$, and the growth condition assumed on $h$ (cf.\ Assumption \ref{ass:casestudy}) have been used. Now, by taking $x$ sufficiently large, we reach a contradiction since $\gamma>1$ by assumption. Hence $\mathcal{S} \neq \emptyset$. 
\end{proof}

\begin{proposition}
\label{prop:freeboundary}
For any $y \in (0,1)$ let
\begin{equation}
\label{freeboundary}
d(y):=\inf\{x>0:\, v(x,y) \geq x\},
\end{equation}
where the convention $\inf \emptyset = + \infty$ has been used. Then
\begin{itemize}
\item[(i)] \begin{equation}
\label{contstopbis}
\mathcal{C}=\{(x,y) \in \mathcal{O}:\, x < d(y)\} \quad \text{and} \quad \mathcal{S}=\{(x,y) \in \mathcal{O}:\, x \geq d(y)\};
\end{equation}
\item[(ii)] $y \mapsto d(y)$ is increasing and left-continuous;
\item[(iii)] there exist $0 < x_{\star} < x^{\star} < \infty$ such that for any $y \in [0,1]$
$$(h')^{-1}\big(\rho - \beta_2) \vee x_{\star} \leq d(y) \leq x^{\star}.$$
\end{itemize}
\end{proposition}

\begin{proof}
\emph{(i)}. To show that \eqref{contstopbis} holds true it suffices to show that if $(x_1,y) \in \mathcal{S}$, then $(x_2,y) \in \mathcal{S}$ for any $x_2 \geq x_1$. Let $\tau^{\varepsilon}:= \tau^{\varepsilon}(x_2,y)$ be an $\varepsilon$-optimal stopping time for $v(x_2,y)$. Then, exploiting the fact that
$\X^{x_2,y}_t = \frac{x_2}{x_1}\X^{x_1,y}_t \geq \X^{x_1,y}_t$ a.s.\ and the monotonicity of $h'$, we can write from \eqref{valueOS-bis}
\begin{eqnarray}
0  & \geq &  v(x_2,y) - x_2 \geq \E\bigg[\int_0^{\tau^{\varepsilon}} e^{- \rho t} \X^{x_2,y}_t\Big(h'\big(\X^{x_2,y}_t \big) - (\rho - \beta_2 - (g_2-g_1)\pi^y_t)\Big) dt\bigg] - \varepsilon \nonumber \\
& \geq & \frac{x_2}{x_1}\,\E\bigg[\int_0^{\tau^{\varepsilon}} e^{- \rho t} \X^{x_1,y}_t\Big(h'\big(\X^{x_1,y}_t \big) - (\rho - \beta_2 - (g_2-g_1)\pi^y_t)\Big) dt\bigg] - \varepsilon \\
& \geq & \frac{x_2}{x_1}\,\big( v(x_1,y) - x_1\big) - \varepsilon = - \varepsilon. \nonumber
\end{eqnarray}
Therefore, by arbitrariness of $\varepsilon$, we conclude that $(x_2,y) \in \mathcal{S}$ as well, and therefore that $d$ as in \eqref{freeboundary} splits $\mathcal{C}$ and $\mathcal{S}$ as in \eqref{contstopbis}.
\vspace{0.1cm}

\emph{(ii)}. Let $(x,y_1)\in \mathcal{C}$. Since $y \mapsto v(x,y)$ is decreasing by Proposition \ref{prop:continuousOS}-(ii), it thus follows that $(x,y_2)\in \mathcal{C}$ for any $y_2 \geq y_1$. This in turn implies that $y \mapsto d(y)$ is increasing. The monotonicity of $y \mapsto d(y)$, together with the fact that $\mathcal{S}$ is closed, then give the claimed left-continuity by standard arguments.
\vspace{0.1cm}

\emph{(iii)}. Let $\Theta^x_t:= x \exp\big\{(\beta_2 - \frac{1}{2}\sigma^2 + (g_2-g_1))t + \sigma I_t\big\}$, and introduce the one-dimensional optimal stopping problem 
\begin{align}
\label{lowerOS}
& v^{\star}(x):= \inf_{\tau \geq 0} \E\bigg[\int_0^{\tau} e^{- \rho t} \Theta^x_t h'(\Theta^x_t) dt + e^{-\rho \tau} \Theta^x_{\tau}\bigg], \quad x>0.
\end{align}
Because $g_2 - g_1 <0$, $h'$ is increasing, and $\pi^y_t \leq 1$ a.s.\ for all $t\geq 0$ and $y\in (0,1)$, it is not hard to see that $v(x,y) \geq v^{\star}(x)$ for any $(x,y) \in \mathcal{O}$.

By arguments similar to those employed to prove \emph{(i)} above one can show that there exists $x^{\star}$ such that $\{x \in (0,\infty):\, v^{\star}(x) \geq x\} = \{x \in (0,\infty):\, x \geq x^{\star}\}$. In fact, by arguing as in the proof of Lemma \ref{lem:Snotempty}, one has that the latter set is not empty. Then the following inclusions hold
\begin{align*}
& \{x \in (0,\infty):\, x \geq x^{\star}\} \subseteq \{(x,y) \in \mathcal{O}:\, v(x,y) \geq x\} = \{(x,y) \in \mathcal{O}:\, x \geq d(y)\},
\end{align*}
which in turn show that $d(y) \leq x^{\star}$ for all $y \in (0,1)$. Hence, also $d(y)\leq x^{\star}$ for all $y \in [0,1]$, by setting $d(0+):=\lim_{y\downarrow 0}d(y)$
by monotonicity, and $d(1):=\lim_{y\uparrow 0}d(y)$ by left-continuity.

As for the lower bound of $d$, notice that \eqref{inclusion} implies
\begin{equation}
\label{lowerbound-b}
d(y) \geq (h')^{-1}\big(\rho - \beta_2 - (g_2-g_1)y\big)=:\zeta(y), \quad y \in (0,1),
\end{equation}
where $(h')^{-1}(\,\cdot\,)$ is the inverse of the strictly increasing function $h': [0,\infty) \mapsto (0,\infty)$ (notice that $\rho - \beta_2 - (g_2-g_1)y \geq0 $ since $\rho>\beta_2$, $g_2-g_1<0$, and $y >0$). Since $(h')^{-1}$ is strictly increasing, and $-(g_2-g_1)y \geq0 $, we can conclude from \eqref{lowerbound-b} that $d(y) \geq (h')^{-1}\big(\rho - \beta_2)$ for every $y\in [0,1]$.

Moreover, setting $\Psi^x_t:= x \exp\{(\beta_2 - \frac{1}{2}\sigma^2)t + \sigma I_t\}$ and introducing the one-dimensional optimal stopping problem 
\begin{align}
\label{upperOS}
& v_{\star}(x):= \inf_{\tau \geq 0} \E\bigg[\int_0^{\tau} e^{- \rho t} \Psi^x_t h'(\Psi^x_t) dt + e^{-\rho \tau} \Psi^x_{\tau}\bigg], \quad x>0,
\end{align}
one has that $v(x,y) \leq v_{\star}(x)$ for any $(x,y) \in \mathcal{O}$. Following arguments as those employed above, the last inequality implies that $d(y) \geq x_{\star}$ for all $y \in [0,1]$, where $x_{\star}:=\inf\{x>0:\, v_{\star}(x) \geq x\} \in (0,\infty)$.
\end{proof}


\subsubsection{Smooth-Fit Property and Continuity of the Free Boundary}
\label{sec:smoothfit}

We now aim at proving further regularity of $v$ and of the free boundary $d$. 
 
The second-order linear elliptic differential operator
\begin{align}
\label{generator}
& \mathbb{L} := \big(\beta_2 + (g_2-g_1)y\big)x \frac{\partial}{\partial x} + \frac{1}{2}\sigma^2 x^2 \frac{\partial^2}{\partial x^2} + \Big(\lambda_2 - (\lambda_1 + \lambda_2)y\Big) \frac{\partial}{\partial y} \nonumber \\
&  + \frac{1}{2} \Big((\alpha_1-\alpha_2)^2 + \frac{(g_2-g_1)^2}{\sigma^2}\Big)y^2(1-y)^2 \frac{\partial^2}{\partial y^2},
\end{align}
acting on any function $f \in C^{2}(\mathcal{O})$, is the infinitesimal generator of the process $(\X,\pi)$. The nondegeneracy of the process $(\X,\pi)$, the smoothness of the coefficients in \eqref{generator}, together with the subharmonic characterization of $v$, allow to prove by standard arguments (see, e.g., \cite{PS}, Ch.\ 3, Sec.\ 7.1) and classical regularity results for elliptic partial differential equations (see, e.g., \cite{GT}) the following result. 
\begin{lemma}
\label{lem:vC2cont}
The value function $v$ of \eqref{valueOS} belongs to $C^{2}$ separately strictly inside $\mathcal{C}$ and $\mathcal{S}$ (i.e.\ away from the boundary $\partial \mathcal{C}$ of $\mathcal{C}$). Moreover, inside $\mathcal{C}$ it uniquely solves
\begin{equation}
\label{PDEcont}
\big(\mathbb{L} - \rho)v(x,y) = - x h'(x),
\end{equation}
with $\mathbb{L}$ as in \eqref{generator}.
\end{lemma}

We continue our analysis by proving that the value function of \eqref{valueOS} belongs to $C^{1}((0,\infty) \times (0,1))$. This will be obtained through probabilistic methods that rely on the regularity (in the sense of diffusions) of the stopping set $\mathcal{S}$ for the process $(\X,\pi)$ (see \cite{DeAP18} where this methodology has been recently developed in a general context; for other examples refer to \cite{DeAGeVi} and \cite{JP2017}). Recall that the boundary points are regular for $\mathcal{S}$ relative to $(\X,\pi)$ if (cf.\ Definition 2.9 p.\ 249  in \cite{KS})
\begin{equation}
\label{regularity-0}
\widehat{\tau}(x_o,y_o):=\inf\{t>0:\, (\X^{x_o,y_o}_t,\pi^{y_o}_t) \in \mathcal{S}\} = 0 \qquad a.s. \quad \forall (x_o,y_o) \in \partial\mathcal{C}.
\end{equation}
The time $\widehat{\tau}(x_o,y_o)$ is the first hitting time of $(\X^{x_o,y_o},\pi^{y_o})$ to $\mathcal{S}$.

Notice that for every bounded Borel function $f: \mathbb{R}^2 \mapsto \mathbb{R}$ one has $\E_{(x,y)}\big[f(\X_t,\pi_t)\big]=\E_{(u,y)}\big[f(e^{U_t},\pi_t)\big]$, where $u:=\ln(x)$ and $U_t:=\ln(\X_t)$ is such that $dU_t=\big(\beta_2 + (g_2-g_1)\pi_t - \frac{1}{2}\sigma^2\big)dt + \sigma dI_t$. Due the nondegeneracy of the process $(U,\pi)$, and the smoothness and boundedness of its coefficients, we have that $(U,\pi)$ has a continuous transition density $\widehat p(\cdot,\cdot,\cdot;u,y)$, $(u,y) \in \mathbb{R} \times (0,1)$, such that for any $t\geq0$ and $(u',y') \in \mathbb{R} \times (0,1)$ (see, e.g., \cite{Aronson})
\begin{eqnarray}
\label{bounds-Gaussian}
&& \frac{M}{t}\exp\Big\{-\lambda\frac{\big((u-u')^2 + (y-y')^2\big)}{t}\Big\} \geq \widehat{p}(t,u',y';u,y) \nonumber \\
&& \geq \frac{m}{t}\exp\Big\{-\Lambda\frac{\big((u-u')^2 + (y-y')^2\big)}{t}\Big\},
\end{eqnarray}
for some constants $M>m>0$ and $\Lambda > \lambda >0$.   
It thus follows that $(u,y) \mapsto \E_{(u,y)}\big[f(e^{U_t},\pi_t)\big]$ is continuous, so that $(U,\pi)$ is a strong Feller process. Hence, $(\X,\pi)$ is strong Feller as well, and we can therefore conclude that \eqref{regularity-0} holds true if and only if (see \cite{DYN}, pp.\ 32-40)
\begin{equation}
\label{regularity1}
\tau^{\star}(x_n,y_n) \rightarrow 0 \quad \text{a.s.} \quad \text{whenever} \quad \mathcal{C} \supseteq (x_n,y_n)_{n} \rightarrow (x_o,y_o) \in \partial \mathcal{C},
\end{equation}
where $\tau^{\star}$ is as in \eqref{OStime}.

The next proposition shows the validity of \eqref{regularity-0}.

\begin{proposition}
\label{prop:regularity}
The boundary points in $\partial \mathcal{C}$ are regular for $\mathcal{S}$ relative to $(\X,\pi)$; that is, \eqref{regularity-0} holds.
\end{proposition}

\begin{proof}
Let $(x_o,y_o)\in \partial \mathcal{C}$, and set $u_o:=\ln(x_o)$. With $U$ as defined above, we set $\widehat{\sigma}(u_o,y_o):=\widehat{\tau}(e^{u_o},y_o)$, $(u_o,y_o) \in \mathbb{R} \times (0,1)$, and we equivalently rewrite \eqref{regularity-0} in terms of the process $(U,\pi)$ as  
$$\widehat{\sigma}(u_o,y_o):=\inf\{t>0:\, U^{u_o,y_o}_t \geq \ln(d(\pi^{y_o})\} = 0\,\, \mbox{a.s.} \,\, \forall (u_o,y_o) \,\, \mbox{such that} \,\, u_o=\ln(d(y_o)).$$

Given that $y \mapsto \ln(d(y))$ is increasing (since $y \mapsto d(y)$ is such), then the region $\widehat{\mathcal{S}}:=\{(u,y) \in \mathbb{R} \times (0,1):\, u \geq \ln(d(y))\}$ enjoys the so-called \emph{cone property} (see \cite{KS}, p.\ 250). In particular, we can always construct a cone $C_o$ with vertex in $(u_o,y_o)$ and aperture $0 \leq \phi \leq \pi/2$ such that $C_o \cap (\mathbb{R} \times (0,1)) \subseteq \widehat{\mathcal{S}}$, and for any $t_o\geq 0$ we can write that
\begin{equation}
\label{estimatereg}
\P(\widehat{\sigma}(u_o,y_o) \leq t_o) \geq \P((U^{u_o,y_o}_{t_o}, \pi^{y_o}_{t_o}) \in C_o).
\end{equation}

Then using \eqref{bounds-Gaussian} one has
\begin{align}
\label{conoinv}
& \P( (U^{u_o,y_o}_{t_o}, \pi^{y_o}_{t_o}) \in C_o) = \int_{C_o} \widehat{p}(t_o,u_o,y_o;u,y) du dy \geq  \int_{C_o} \frac{m}{t_o}e^{-\Lambda\frac{((u-u_o)^2 + (y-y_o)^2)}{t_o}} du dy \nonumber \\
& = m \int_{C_o} e^{-\Lambda \big((u')^2 + (y')^2\big)} du' dy' =:\ell > 0, 
\end{align}
where we have used that the change of variable $u':= (u-u_o)/\sqrt{t_o}$ and $y':= (y-y_o)/\sqrt{t_o}$ maps the cone $C_o$ into itself. The number $\ell$ above depends on $u_o,y_o$, but it is independent of $t_o$. From \eqref{estimatereg} and \eqref{conoinv} we thus have that $\P(\widehat{\sigma}(u_o,y_o) \leq t_o) \geq \ell$, and letting $t_o \downarrow 0$ we obtain $\P(\widehat{\sigma}(u_o,y_o) = 0) \geq \ell > 0$. However, $\{\widehat{\sigma}(u_o,y_o) = 0\} \in \mathcal{H}_{0}$, and by the Blumenthal's 0-1 Law we obtain $\P(\widehat{\sigma}(u_o,y_o) = 0) =1$, which completes the proof.
\end{proof}

\begin{theorem}
\label{thm:C1}
One has that $v \in C^1(\mathcal{O})$. 
\end{theorem}

\begin{proof}
The value function belongs to $C^{2}$ strictly inside the continuation region due to Lemma \ref{lem:vC2cont}, and it is $C^{\infty}$ strictly inside the stopping region where $v=x$. It thus only remains to prove that $v$ is continuously differentiable across $\partial \mathcal{C}$. In the following, we will prove that: \emph{(i)} the function $\overline{w}:=\frac{1}{x}(v - x)$ has continuous derivative with respect to $x$ across $\partial \mathcal{C}$ (and this clearly implies the continuity of $v_x$ across $\partial \mathcal{C}$); \emph{(ii)} that the function $v_y$ is continuous across $\partial \mathcal{C}$.
\vspace{0.1cm}

\emph{(i)\, Continuity of $v_x$ across $\partial \mathcal{C}$}. For the subsequent arguments it is useful to notice that the function $\overline{w}=\frac{1}{x}(v-x)$ admits the representation (recall \eqref{valueOS-bis})
\begin{equation}
\label{overlinew}
\overline{w}(x,y)=\inf_{\tau \geq 0}\E\bigg[\int_0^{\tau} e^{-\rho t} \X^{1,y}_t\Big(h'\big(\X^{x,y}_t\big) - \big(\rho - \beta_2 -(g_2-g_1)\pi^y_s\Big) ds\bigg],
\end{equation}
and to bear in mind that the optimal stopping time $\tau^{\star}$ for $v$ as in \eqref{OStime} is also optimal for $\overline{w}$ since $v \geq x$ if and only if $\overline{w}\geq 0$. We now prove that $\overline{w}_x$ is continuous across $\partial \mathcal{C}$, thus implying continuity of $v_x$ across $\partial \mathcal{C}$.

Take $(x,y) \in \mathcal{C}$, and let $\varepsilon>0$ be such that $x-\varepsilon>0$. Since $x \mapsto \overline{w}(x,y)$ is increasing (due to the monotonicity of $h'$) it is clear that $(x-\varepsilon,y) \in \mathcal{C}$ as well. Denote by $\tau^{\star}_{\varepsilon}(x,y):=\tau^{\star}(x-\varepsilon,y)$ the optimal stopping time for $\overline{w}(x-\varepsilon,y)$, and notice that $\tau^{\star}_{\varepsilon}(x,y)$ is suboptimal for $\overline{w}(x,y)$ and $\tau^{\star}_{\varepsilon}(x,y)\rightarrow \tau^{\star}(x,y)$ a.s. To simplify exposition in the following we write $\tau^{\star}_{\varepsilon}:=\tau^{\star}_{\varepsilon}(x,y)$ and $\tau^{\star}:=\tau^{\star}(x,y)$. We can then write from \eqref{overlinew}
\begin{align}
\label{wxcont-1}
& 0 \leq \frac{\overline{w}(x,y) - \overline{w}(x-\varepsilon,y)}{\varepsilon} \leq \frac{1}{\varepsilon}\E\bigg[\int_0^{\tau^{\star}_{\varepsilon}} e^{-\rho t} \X^{1,y}_t\Big(h'\big(\X^{x,y}_t\big) - h'\big(\X^{x-\varepsilon,y}_t\big)\Big) dt \bigg]\nonumber \\
& = \E\bigg[\int_0^{\tau^{\star}_{\varepsilon}} e^{-\rho t} (\X^{1,y}_t)^2 h''\big(\X^{\xi_{\varepsilon},y}_t\big) dt \bigg], \nonumber
\end{align}
for some $\xi_{\varepsilon} \in (x-\varepsilon,x)$, and where in the last step we have used the mean value theorem, and the fact that $\X^{x,y}_t - \X^{x-\varepsilon,y}_t = \varepsilon \X^{1,y}_t$.
Letting $\varepsilon \downarrow 0$, invoking the dominated convergence theorem (thanks to the fact that $\rho>\big(\gamma\beta_2 + \frac{1}{2}\sigma^2\gamma(\gamma-1)\big) \vee \big(2\beta_2 + \sigma^2\big)$ by Assumption \ref{ass:rho}), and using that $\overline{w} \in C^1(\mathcal{C})$ (since $v \in C^1(\mathcal{C})$), we then find from the latter that
\begin{equation}
\label{wxcont-2}
0 \leq \overline{w}_x(x,y) \leq \E\bigg[\int_0^{\tau^{\star}} e^{-\rho t} (\X^{1,y}_t)^2 h''\big(\X^{x,y}_t\big) dt \bigg].
\end{equation}
Let now $(x_o,y_o)$ be any arbitrary point belonging to $\partial \mathcal{C}$. Taking limits in \eqref{wxcont-2} as $(x,y) \rightarrow (x_o,y_o)$, by the dominated convergence theorem and thanks to Proposition \ref{prop:regularity} we obtain that 
$$0 \leq \liminf_{(x,y) \rightarrow (x_o,y_o) \in \partial \mathcal{C}}\overline{w}_x(x,y)  \leq \limsup_{(x,y) \rightarrow (x_o,y_o) \in \partial \mathcal{C}}\overline{w}_x(x,y) \leq 0,$$ 
thus proving that $\overline{w}_x$ is continuous across $\partial \mathcal{C}$. This immediately implies the continuity of $v_x$ across $\partial \mathcal{C}$, upon recalling that $v=x(\overline{w}+1)$.
\vspace{0.1cm}

\emph{(ii)\, Continuity of $v_y$ across $\partial \mathcal{C}$}. Take again $(x,y) \in \mathcal{C}$, and let $\varepsilon>0$ be such that $y+\varepsilon<1$. Since $y \mapsto v(x,y)$ is decreasing (cf.\ Proposition \ref{prop:continuousOS}-(ii)), it is clear that $(x,y+\varepsilon) \in \mathcal{C}$ as well. Denote by $\tau^{\star}_{\varepsilon}(x,y):= \tau^{\star} (x, y+\varepsilon)$ the optimal stopping time for $v(x, y +\varepsilon)$ and notice that $\tau^{\star}_{\varepsilon}(x,y)$ is suboptimal for $v(x,y)$ and $\tau^{\star} (x, y+\varepsilon) \rightarrow \tau^{\star} (x,y) $ a.s.\ as $\varepsilon \downarrow 0$. In order to simplify the notation, in the following we write $\tau^{\star}_{\varepsilon}$ instead of $\tau^{\star}_{\varepsilon}(x,y)$.

From Proposition \ref{prop:continuousOS}-(ii) and \eqref{valueOS-bis} we can then write
\begin{align*}
\label{}
0  \ge &\frac{v(x,y+\varepsilon) - v(x,y)}{\varepsilon} \ge  \frac{1}{\varepsilon} \ \E\bigg[\int_0^{\tau^{\star}_{\varepsilon}} e^{- \rho t} \X^{x,y+\varepsilon}_t  \Big[  h'\big( \X^{x,y+\varepsilon}_t \big) -  \left( \rho - \beta_2 -  \pi^{y+\varepsilon}_t (g_2 - g_1)   \right) \Big] dt \bigg] \nonumber   \\   
& - \frac{1}{\varepsilon} \ \E \bigg[\int_0^{\tau^{\star}_{\varepsilon}} e^{- \rho t} \X^{x,y}_t  \Big[  h'\big( \X^{x,y}_t \big) -  \big( \rho - \beta_2 -  \pi^{y}_t   (g_2 - g_1) \big) \Big] dt \bigg]  \nonumber   \\ 
= &  \frac{1}{\varepsilon} \left( \E\bigg[ \int_0^{\tau^{\star}_{\varepsilon}} e^{- \rho t} \left[ \X^{x,y+\varepsilon}_t  h'\big( \X^{x,y+\varepsilon}_t \big) -  \X^{x,y}_t  h'\big( \X^{x,y}_t \big) \right] dt - \int_0^{\tau^{\star}_{\varepsilon}} e^{- \rho t} (\rho - \beta_2) \left(  \X^{x,y+\varepsilon}_t - \X^{x,y}_t  \right) dt \bigg] \right) \nonumber   \\  
& + \frac{1}{\varepsilon} \ \E\bigg[\int_0^{\tau^{\star}_{\varepsilon}} e^{- \rho t} (g_2 - g_1) \left(  \X^{x,y+\varepsilon}_t \pi^{y+\varepsilon}_t - \X^{x,y}_t \pi^{y}_t \right) dt \bigg].  \nonumber   
\end{align*}
Now, add and subtract both $\E[\int_0^{\tau^{\star}_{\varepsilon}} e^{- \rho t}\X^{x,y+\varepsilon}_t  h'( \X^{x,y}_t) dt]$ and $(g_2-g_1)\E[\int_0^{\tau^{\star}_{\varepsilon}} e^{- \rho t} \X^{x,y+\varepsilon}_t \pi^{y}_t dt]$ in the right-hand side of the latter, and recall that $(g_2 - g_1) <0$, that $\X^{x,y}_t  \ge 0$ a.s.\ for every $t \ge 0$, as well as that $(\pi^{y+\varepsilon}_t - \pi^{y}_t) \ge 0$ a.s.\ for every $t \ge 0$. Then, after rearranging terms and employing the integral mean value theorem (for some $L_t^{\varepsilon} \in (\X^{x,y+\varepsilon}_t, \X^{x,y}_t)$ a.s.), we obtain from the equation above that
\begin{align}
\label{stimaC1y-a}
0  \ge &\frac{v(x,y+\varepsilon) - v(x,y)}{\varepsilon} \geq \frac{1}{\varepsilon} \ \E \bigg[ \int_0^{\tau^{\star}_{\varepsilon}} e^{- \rho t} \X^{x,y+\varepsilon}_t  \Big[  h'\big( \X^{x,y+\varepsilon}_t \big)  -   h'\big( \X^{x,y}_t \big)  \Big] dt \bigg]  \nonumber   \\   
& + \frac{1}{\varepsilon} \ \E \bigg[\int_0^{\tau^{\star}_{\varepsilon}} e^{- \rho t} \Big(  \X^{x,y+\varepsilon}_t  -   \X^{x,y}_t   \Big)  \Big[  h'\big( \X^{x,y}_t \big) -  \Big( \rho - \beta_2 -  \pi^{y}_t (g_2 - g_1)  \Big) \Big] dt \bigg]  \nonumber   \\ 
& - \frac{1}{\varepsilon} \ | g_2 - g_1| \ \E \bigg[\int_0^{\tau^{\star}_{\varepsilon}} e^{- \rho t}  \X^{x,y+\varepsilon}_t   \left( \pi^{y+\varepsilon}_t - \pi^{y}_t  \right)  dt \bigg]    \\ 
\geq & \frac{1}{\varepsilon} \ \E \bigg[\int_0^{\tau^{\star}_{\varepsilon}} e^{- \rho t} \Big(  \X^{x,y+\varepsilon}_t  -   \X^{x,y}_t   \Big)  \Big(\X^{x,y+\varepsilon}_t   h'' \big( L_t^{\varepsilon}\big) + h'\big( \X^{x,y}_t \big) \Big) dt \bigg]  \nonumber   \\ 
& - \frac{1}{\varepsilon} \ | g_2 - g_1| \ \E \bigg[ \int_0^{\tau^{\star}_{\varepsilon}} e^{- \rho t}  \X^{x,y+\varepsilon}_t   \left( \pi^{y+\varepsilon}_t - \pi^{y}_t  \right)  dt \bigg]. \nonumber 
\end{align}
In the last inequality we have used that $\rho - \beta_2 - \pi^{y}_t (g_2 - g_1) \geq 0$, since $\rho>\beta_2$ by Assumption \ref{ass:rho}, that $g_2-g_1<0$, and that $\X^{x,y+\varepsilon}_t  \leq   \X^{x,y}_t$. 

Define now $\Delta \pi^y_t:=\frac{1}{\varepsilon}(\pi^{y+\varepsilon}_t - \pi^{y}_t)$, $t\geq0$, and notice that, by using the second equation in \eqref{syst:XYhat}, we can write
\begin{equation*}
\label{Deltapi}
d\Delta\pi^y_t = -(\lambda_1 + \lambda_2) \Delta\pi^y_t dt + \Delta\pi^y_t\Big(1- \pi^{y+\varepsilon}_t - \pi^y_t\Big)\Big[\frac{(g_2-g_1)}{\sigma}dI_t + (\alpha_1-\alpha_2)dI^1_t\Big], \quad t>0, 
\end{equation*}
with $\Delta\pi^y_0=1$. With the help of It\^o's formula, it can be easily shown that 
\begin{equation}
\label{solDeltapi}
\Delta\pi^y_t = \exp\Big\{-(\lambda_1 + \lambda_2)t -\theta^2\int_0^t \big(1- \pi^{y+\varepsilon}_s - \pi^y_s\big)^2ds + \int_0^t \big(1- \pi^{y+\varepsilon}_s - \pi^y_s\big)\Big[\frac{(g_2-g_1)}{\sigma}dI_s + (\alpha_1-\alpha_2)dI^1_s\Big]\Big\},
\end{equation}
with $\theta^2:=\frac{1}{2}\big[\frac{(g_2-g_1)^2}{\sigma^2} + (\alpha_1-\alpha_2)^2\big]$, solves the previous stochastic differential equation.

Also, by \eqref{sol:X} and simple algebra, 
\begin{equation}
\label{DeltaX}
\frac{1}{\varepsilon}  \left(  \X^{x,y+\varepsilon}_t  - \X^{x,y}_t   \right) 
= \X^{x,y}_ t \left(\frac{ e^{\varepsilon(g_2-g_1) \int_0^t \Delta\pi^{y}_s ds}  -  1 }{\varepsilon} \right).
\end{equation}

Employing the definition of $\Delta \pi^y_t$ and \eqref{DeltaX} in \eqref{stimaC1y-a}, and using that $\X^{x,y+\varepsilon}_t \leq \X^{x,y}_t$, one finds
\begin{align}
\label{ineq3}
0  \ge &\frac{v(x,y+\varepsilon) - v(x,y)}{\varepsilon} \geq \E \bigg[\int_0^{\tau^{\star}_{\varepsilon}} e^{- \rho t} \X^{x,y}_t  \Big( \frac{ e^{\varepsilon (g_2-g_1) \int_0^t \Delta \pi^{y}_s ds}  -  1 }{\varepsilon} \Big) \cdot \nonumber \\   
&  \cdot  \Big( \X^{x,y}_t   h'' \big( L_t^{\varepsilon}\big)  +  h'\big( \X^{x,y}_t \big)\Big)  dt \bigg] - | g_2 - g_1| \ \E\bigg[\int_0^{\tau^{\star}_{\varepsilon}} e^{- \rho t}  \X^{x,y}_t   \Delta\pi^y_t  dt\bigg].    
\end{align}

We now aim at taking limits as $\varepsilon \downarrow 0$ in \eqref{ineq3}. To this end, notice that $\Delta \pi^y_t \rightarrow Z^y_t$ a.s.\ for all $t\geq0$, as $\varepsilon \downarrow 0$, where, by Theorem $39$ in Chapter V.7 of \cite{Pr04}, $(Z^y_t)_{t\geq0}$ is the unique strong solution to 
\begin{equation}
\label{Z:sde}
d  Z_t^{y} = - (\lambda_1+\lambda_2) Z_t^{y}  dt + Z_t^{y} (1 - 2 \pi_s^y ) \left[\frac{(g_2-g_1)}{\sigma}dI_t + (\alpha_1-\alpha_2)dI^1_t\right], \quad t>0,
\end{equation}
with $Z_0^{y}  =  1$. Then, if we were allowed to invoke the dominated convergence theorem when taking limits as $\varepsilon \downarrow 0$ in \eqref{ineq3}, we would obtain that
\begin{align}
\label{ineq-final}
0  \ge &\, v_y(x,y) \geq (g_2-g_1)\E \bigg[\int_0^{\tau^{\star}} e^{- \rho t} \X^{x,y}_t \Big(\int_0^t Z^{y}_s ds\Big) \Big(  \X^{x,y}_t   h'' \big(\X^{x,y}_t\big)  +  h'\big(\X^{x,y}_t \big) \Big)  dt \bigg] \nonumber \\
& - | g_2 - g_1| \ \E\bigg[\int_0^{\tau^{\star}} e^{- \rho t}  \X^{x,y}_t   Z^y_t  dt\bigg],  
\end{align}
upon recalling that $v\in C^{2}(\mathcal{C})$. Therefore, letting $(x_o,y_o)$ be any arbitrary point belonging to $\partial \mathcal{C}$, by taking limits in \eqref{ineq-final} as $(x,y) \rightarrow (x_o,y_o)$, by the dominated convergence theorem and thanks to Proposition \ref{prop:regularity} we obtain that 
$$0 \geq \limsup_{(x,y) \rightarrow (x_o,y_o) \in \partial \mathcal{C}}v_y(x,y) \geq  \liminf_{(x,y) \rightarrow (x_o,y_o) \in \partial \mathcal{C}}v_y(x,y) \geq 0,$$ thus proving that $v_y$ is continuous across $\partial \mathcal{C}$.

In order to complete the proof it thus only remains to show that the dominated convergence theorem can indeed be applied when taking limits as $\varepsilon \downarrow 0$ in \eqref{ineq3}. This is what we are going to show in the two following technical steps.
\vspace{0.1cm}

\emph{Step 1.} To prove that the dominated convergence theorem can be invoked when taking $\varepsilon \downarrow 0$ in the first expectation on the right-hand side of \eqref{ineq3}, we set
$$\Lambda_{\varepsilon}:= \int_0^{\tau^{\star}_{\varepsilon}} e^{- \rho t} \X^{x,y}_t \Big( \frac{ e^{\varepsilon (g_2-g_1) \int_0^t \Delta \pi^{y}_s ds}  -  1 }{\varepsilon} \Big)\Big(  \X^{x,y}_t   h''\big( L_t^{\varepsilon}\big)  +  h'\big( \X^{x,y}_t \big)\Big)  dt,$$ 
and we show that the family of random variables $\{\Lambda_{\varepsilon}, \varepsilon \in (0,1-y)\}$ is bounded in $L^2(\Omega,\mathcal{F},\P)$, hence uniformly integrable.

Notice that by Assumption \ref{ass:casestudy}-(ii) and the fact that $L_t^{\varepsilon} \leq \X^{x,y}_t$ a.s., one has a.s.\ for any $t\geq0$
$$\X^{x,y}_t\Big[\X^{x,y}_t  h'' \big(L_t^{\varepsilon}\big)  +  h' \big(\X^{x,y}_t\big) \Big] \leq  \widehat{K}\Big(1 + \big(\X^{x,y}_t\big)^{\gamma \vee 2}\Big),$$
for some constant $\widehat{K}>0$ (independent of $\varepsilon$), so that by Jensen's inequality 
\begin{align*}
& \big|\Lambda_{\varepsilon}\big|^2 \leq \frac{2 \widehat{K}^2}{\rho^2} \int_0^{\infty} \rho e^{- \rho t} \Big(\frac{1 - e^{\varepsilon (g_2-g_1) \int_0^t \Delta \pi^{y}_s ds}}{\varepsilon} \Big)^2\Big(1 + \big(\X^{x,y}_t\big)^{2\gamma \vee 4}\Big) dt. \nonumber 
\end{align*}
Then, taking expectations and applying H\"older's inequality
\begin{align}
\label{stimaC1y-B}
& \E\Big[\big|\Lambda_{\varepsilon}\big|^2\Big]^{\frac{1}{2}} \leq K'\E\bigg[\int_0^{\infty} e^{-\rho t} \Big(\frac{1- e^{\varepsilon (g_2-g_1) \int_0^t \Delta \pi^{y}_s ds}}{\varepsilon} \Big)^4 dt\bigg]^{\frac{1}{4}}\E\bigg[\int_0^{\infty} e^{- \rho t} \Big(1 + \big(\X^{x,y}_t\big)^{4\gamma \vee 8}\Big) dt\bigg]^{\frac{1}{4}},
\end{align}
for some other constant $K'>0$, independent of $\varepsilon$, that in the following will be varying from line to line.

The standard inequality $1- e^{-x} \leq x$, with $x=\varepsilon (g_1-g_2) \int_0^t \Delta \pi^{y}_s ds \geq 0$, allows us to continue from \eqref{stimaC1y-B} and write
\begin{align}
\label{stimaC1y-C}
& \E\Big[\big|\Lambda_{\varepsilon}\big|^2\Big]^{\frac{1}{2}} \leq K'\E\bigg[\int_0^{\infty} e^{-\rho t} \Big(\int_0^t \Delta \pi^{y}_s ds\Big)^4 dt\bigg]^{\frac{1}{4}}\E\bigg[\int_0^{\infty} e^{- \rho t} \Big(1 + \big(\X^{x,y}_t\big)^{4(\gamma \vee 2)}\Big) dt\bigg]^{\frac{1}{4}}. 
\end{align}

We now treat the two expectations in \eqref{stimaC1y-C} separately. First of all, notice that by Jensen's inequality
\begin{equation}
\label{jensen}
\Big(\int_0^t \Delta \pi^{y}_s ds\Big)^4 = \Big(\frac{1}{t}\int_0^t t \,\Delta \pi^{y}_s ds\Big)^4\leq t^3 \int_0^t \big(\Delta \pi^{y}_s\big)^4 ds.
\end{equation}
Second of all, thanks to the nonnegativity of $(\Delta\pi^y)^4$, we can invoke Fubini-Tonelli's theorem and using also \eqref{jensen}, obtain 
\begin{align}
\label{Fubini1}
& \E\bigg[\int_0^{\infty} e^{-\rho t} \Big(\int_0^t \Delta \pi^{y}_s ds\Big)^4 dt\bigg] \leq \E\bigg[\int_0^{\infty} e^{-\rho t} t^3 \Big(\int_0^t \big(\Delta \pi^{y}_s\big)^4 ds\Big) dt\bigg] \nonumber \\
& = \frac{1}{\rho^4}\int_0^{\infty} e^{-\rho s}\big(\rho^3s^3+3\rho^2 s^2+ 6\rho s + 6\big) \E\Big[\big(\Delta \pi^{y}_s\big)^4\Big] ds. 
\end{align}
We now aim at evaluating the expectation in the last integral above.

To accomplish that, notice that by applying It\^o's formula to the process $\xi^y_t:=(\Delta\pi^y_t)^4$, and using \eqref{solDeltapi}, we have for any $t>0$
\begin{equation*}
\label{dynxi}
d\xi^y_t = \xi^y_t\big(-(\lambda_1+\lambda_2) + 12\theta^2(1-\pi^{y+\varepsilon}_t - \pi^y_t)^2\big)dt + 4\xi^y_t(1-\pi^{y+\varepsilon}_t - \pi^y_t)\big[\frac{(g_2-g_1)}{\sigma}dI_s + (\alpha_1-\alpha_2)dI^1_s\big],
\end{equation*}
with $\xi^y_0=1$ and $\theta^2=\frac{1}{2}\big[\frac{(g_2-g_1)^2}{\sigma^2} + (\alpha_1-\alpha_2)^2\big]$. Because $(1-\pi^{y+\varepsilon}_t - \pi^y_t)^2 \leq 2$ a.s.\ for all $t\geq0$, and 
$$\xi^y_t = e^{-(\lambda_1 + \lambda_2)t + 12\theta^2\int_0^t(1-\pi^{y+\varepsilon}_t - \pi^y_t)^2 ds}M^y_t,$$
where $(M^y_t)_{t\geq0}$ is an exponential martingale, it is easy to see that
\begin{equation}
\label{expDeltapi4}
\E\big[(\Delta\pi^y_t)^4\big] \leq e^{-(\lambda_1 + \lambda_2)t + 24\theta^2 t}, \quad t\geq0.
\end{equation}
Using the latter estimate in \eqref{Fubini1}, together with Assumption \ref{ass:rho}, we deduce that 
\begin{equation}
\label{uni-1}
\sup_{\varepsilon\in (0,1-y)}\E\bigg[\int_0^{\infty} e^{-\rho t} \Big(\int_0^t \Delta \pi^{y}_s ds\Big)^4 dt\bigg] < \infty.
\end{equation}

As for the second expectation in \eqref{stimaC1y-C}, Assumption \ref{ass:rho} and standard estimates employing \eqref{sol:X} (together with the fact that $(g_2-g_1)\int_0^t \pi^y_s ds <0$) guarantee that it is finite. Moreover, it is independent of $\varepsilon$. Combining this with \eqref{uni-1} we thus find from \eqref{stimaC1y-C} that
$\sup_{\varepsilon \in (0,1-y)}\E\Big[\big|\Lambda_{\varepsilon}\big|^2\Big]^{\frac{1}{2}} < \infty,$
thus implying that the family of random variables $\{\Lambda_{\varepsilon}, \varepsilon \in (0,1-y)\}$ is bounded in $L^2(\Omega,\mathcal{F},\P)$, hence uniformly integrable.
\vspace{0.1cm}

\emph{Step 2.} We consider the second expectation on the right-hand side of \eqref{ineq3}, and setting
$$\Xi_{\varepsilon}:=\int_0^{\tau^{\star}_{\varepsilon}} e^{- \rho t}  \X^{x,y}_t   \Delta\pi^y_t  dt,$$
we aim at proving that the family of random variables $\{\Xi_{\varepsilon}, \varepsilon \in (0,1-y)\}$ is bounded in $L^2(\Omega,\mathcal{F},\P)$, hence uniformly integrable.

By Jensen's inequality first, and H\"older's inequality then, one finds that
\begin{equation}
\label{stimaexp2-1}
\E\Big[\big|\Xi_{\varepsilon}|^2\Big]^{\frac{1}{2}} \leq \widehat{K}\E\bigg[\int_0^{\infty}e^{-\rho t} \big(\X^{x,y}_t \big)^4 dt\bigg]^{\frac{1}{4}} \E\bigg[\int_0^{\infty}e^{-\rho t} \big(\Delta \pi^y_t \big)^4 dt\bigg]^{\frac{1}{4}},
\end{equation}
for some $\widehat{K}>0$, independent of $\varepsilon$.

The first expectation on the right-hand side of \eqref{stimaexp2-1} is finite thanks to Assumption \ref{ass:rho} and standard estimates employing \eqref{sol:X} (together with the fact that $(g_2-g_1)\int_0^t \pi^y_s ds <0$). Moreover, it is independent of $\varepsilon$.

As for the second one, by interchanging expectation and time integral by Fubini-Tonelli's theorem, and using \eqref{expDeltapi4}, we obtain
$$\E\bigg[\int_0^{\infty}e^{-\rho t} \big(\Delta \pi^y_t \big)^4 dt\bigg]^{\frac{1}{4}} \leq \frac{1}{(\rho + \lambda_1 + \lambda_2 - 24\theta^2)^{\frac{1}{4}}},$$
due to Assumption \ref{ass:rho}. We therefore conclude that (cf.\ \eqref{stimaexp2-1}) $\sup_{\varepsilon \in (0,1-y)}\E\Big[\big|\Xi_{\epsilon}|^2\Big]^{\frac{1}{2}} < \infty$, thus completing the proof.
\end{proof}

The previous theorem in particular implies the so-called \emph{smooth-fit} property, a well known optimality principle in optimal stopping theory. Moreover, by standard arguments based on the strong Markov property of $(\X,\pi)$ (see Chapter III in \cite{PS}) it follows from the results collected so far that the couple $(v,d)$ solves the free-boundary problem
\begin{equation}
\label{FBP}
\left\{
\begin{array}{ll}
\big(\mathbb{L} - \rho\big)v(x,y) =  - xh'(x) & \text{on} \quad \mathcal{C},\\[+6pt]
v(x,y) = x & \text{on} \quad \mathcal{S},\\[+6pt]
v_x(x,y) = 1 & \text{at} \quad x= d(y),\,\,y \in (0,1), \\[+6pt]
v_y(x,y) = 0 & \text{at} \quad x= d(y),\,\,y \in (0,1), 
\end{array}
\right.
\end{equation}
with $v \in C^{2}(\mathcal{C})$.

An important consequence of Theorem \ref{thm:C1} is the following.
\begin{proposition}
\label{thm:fbLip}
One has that $y \mapsto d(y)$ is continuous on $[0,1]$.
\end{proposition}

\begin{proof}
Define the probability measure $\PP$ on $(\Omega,\mathcal{F})$ such that $\frac{d\PP}{d\P}\Big|_{\mathcal{F}_t} = e^{-\frac{1}{2}\sigma^2 t + \sigma I_t}$, $t \geq 0$.
Such a measure is equivalent to $\P$ on $\mathcal{F}_t$, and defining $\widehat I_t:=I_t - \sigma t$, by Girsanov's theorem the latter is a standard $\mathbb{H}$-Brownian motion under $\PP$.

By a change of measure (see, e.g., Section 12 in Chapter IV of \cite{PS}) it is then not difficult to see that $v$ as in \eqref{valueOS-bis} is such that $v(x,y) := x - \widehat{V}(x,y)$, where, for any $(x,y) \in \mathcal{O}$, we have set 
\begin{equation}
\label{def:Vhat}
\widehat{V}(x,y):= \sup_{\tau \geq 0} \widehat{\E}_{(x,y)}\bigg[\int_0^{\tau} e^{-(\rho - \beta_2)t + (g_2-g_1)\int_0^t \pi_s ds } \widehat{H}(\X_t,\pi_t) dt\bigg],
\end{equation}
with $ \widehat{H}(x,y):= \Big(\rho - \beta_2 -(g_2-g_1)y - h'(x)\Big)$. 
In \eqref{def:Vhat} above $\widehat{\E}_{(x,y)}$ denotes the expectation conditioned on the fact that $(\X_0,\pi_0)=(x,y)$ $\PP$-a.s. Since $\{(x,y) \in \mathcal{O}:\, v(x,y) \geq x\} = \{(x,y) \in \mathcal{O}:\, \widehat{V}(x,y) \leq 0\}$, $d(\,\cdot\,)$ is the optimal stopping boundary for the problem with value $\widehat{V}$ as well. 

In order to prove the continuity of $d(\,\cdot\,)$, we now aim at applying Theorem 10 in \cite{Peskir2018} for problem \eqref{def:Vhat}. Notice that $\widehat{V}_x \leq 0$ on $\mathcal{O}$ since $x \mapsto h(x)$ is strictly convex. Moreover, recalling $\theta^2=\frac{1}{2}[(\alpha_1-\alpha_2)^2 + \frac{(g_2-g_1)^2}{\sigma^2}]$, we have $\partial_x\left(\frac{\widehat{H}}{\theta^2 y^2(1-y)^2}\right)<0$ on $\mathcal{O}$ thanks, again, to the strict convexity of $h$.
Also, $\widehat{V}_y$ is continuous across the boundary, due to the $C^1$-property shown in Theorem \ref{thm:C1} for $v=x-\widehat{V}$; hence, the horizontal smooth-fit property holds.

We can therefore apply Theorem 10 of \cite{Peskir2018} (upon noticing that in \cite{Peskir2018} $x$ is the horizontal axis and $y$ is the vertical one, while, in our paper, $x$ is the vertical axis and $y$ is the horizontal one), and conclude that $d$ cannot have discontinuities of the first kind at any point $y\in [0,1)$.
Finally, $d$ is also continuous at $y=1$ since it is left-continuous by Proposition \ref{prop:freeboundary}-(ii).
\end{proof}


\subsection{The Optimal Control for Problem (P3)}
\label{sec:optcontr}

In this section, we provide the form of the optimal debt reduction policy. It is given in terms of the free boundary studied in the previous section.

For $d$ as in \eqref{freeboundary}, introduce under $\P_{(x,y)}$ the nondecreasing process
\begin{equation}
\label{OC-candidate}
\displaystyle \overline{\nu}^{\star}_t = \Big[x - \inf_{0\leq s \leq t}\Big(d\big(\pi_s\big)e^{-(\beta_2 - \frac{1}{2}\sigma^2)s - \sigma I_s - (g_2-g_1)\int_0^s \pi_u du}\Big)\Big] \vee 0, \quad t \geq 0,
\end{equation}
with $\overline{\nu}^{\star}_{0^-}=0$,  and then the process
\begin{equation} 
\label{nustar}
\nu^{\star}_t:=\int_{0}^t e^{-(\beta_2-\frac{1}{2}\sigma^2)s - \sigma I_s - (g_2-g_1)\int_0^s \pi_u du} d\overline{\nu}^{\star}_s, \quad t \geq 0, \qquad \nu^{\star}_{0^-}=0.
\end{equation}

Notice that since $\overline{\nu}^{\star}_t \leq x$ a.s.\ for all $t\geq 0$, and $t \mapsto \overline{\nu}^{\star}_t$ is nondecreasing, it does follows from \eqref{nustar} that $\nu^{\star}$ is admissible. Moreover, $t \mapsto \overline{\nu}^{\star}_t$ is continuous (with the exception of a possible initial jump at initial time), due to the continuity of $y \mapsto d(y)$ and to that of $t \mapsto I_t$, $t \mapsto \pi_t$, and $t \mapsto \int_0^t \pi_s ds$.

\begin{theorem}
\label{teo:ver-2}
Let $\widetilde{V}(x,y):= \int_{0}^{x}\frac{1}{z}v(z,y)dz$, $(x,y)\in [0,\infty) \times [0,1]$. Then one has $\widetilde{V} = V$ on $[0,\infty) \times [0,1]$, and $\nu^{\star}$ as in \eqref{nustar} is optimal for Problem \textbf{(P3)}.
\end{theorem}
\begin{proof}

Recall $U=U_0$ as in \eqref{valueOS-general}, and notice that in our Markovian setting one has $\frac{1}{z}v(z,y) = U(z)$. By the proof of Theorem \ref{teo:verifico} it suffices to show that the right-continuous inverse of the stopping time $\tau^{\star}(z,y)= \ \inf\{t\geq 0 \ | \ \X^{z,y}_t \geq d(\pi^y_t) \}$ (which is optimal for $v(z,y)$, cf.\ \eqref{OStime}) coincides (up to a null set) with $\overline{\nu}^{\star}$.

Then, recall \eqref{taunumeno} from the proof of Theorem \ref{teo:verifico}, fix $(x,y) \in (0,\infty) \times (0,1)$, take $t\geq 0$ arbitrary, and notice that by \eqref{OStime} we have $\P_{(z,y)}$-a.s.~the equivalences
\begin{align*}
&\tau^{\star}(z,y) \leq t  \  \Longleftrightarrow \ \X_{\theta} \geq d(\pi_{\theta}) \ \mbox{ for some } \theta \in [0,t] \ \Longleftrightarrow \\
& z \geq e^{-(\beta_2 - \frac{1}{2}\sigma^2)\theta - \sigma I_{\theta} - (g_2-g_1)\int_0^{\theta}\pi_u du} d(\pi_{\theta}) \ \mbox{ for some } \theta \in [0,t] \Longleftrightarrow \\
& \Big[x - \inf_{0\leq s \leq t}\Big( d(\pi_s)e^{-(\beta_2 - \frac{1}{2}\sigma^2)s - (g_2-g_1)\int_0^{s}\pi_u du - \sigma I_{s}}\Big)\Big] \vee 0 \geq x - z\ \Longleftrightarrow\ \overline{\nu}^{\star}_t \geq x - z \Longleftrightarrow \\
& \tau^{\overline{\nu}^{\star}}_{+}(z) \leq t.
\end{align*}
Hence, $\tau^{\overline{\nu}^{\star}}_{+}(z)=\tau^{\star}(z,y)$ a.s., and $\overline{\nu}^{\star}_{\cdot}$ is the right-continuous inverse of $\tau^{\star}(\cdot,y)$. Since $\overline{\nu}^{\star}$ is admissible, by arguing as in \emph{Step 2} of the proof of Theorem \ref{teo:verifico} the claim follows.
\end{proof}

Notice that the equation of $X^{x,y,\nu}$ in the formulation of Problem {\textbf{(P3)}}, and \eqref{nustar}, yield
$$X^{x,y,\nu^{\star}}_t = e^{(\beta_2-\frac{1}{2}\sigma^2)t + (g_2-g_1)\int_0^t \pi^y_s ds + \sigma I_t}\big[x - \overline{\nu}^{\star}_t\big],$$
which, with regard to \eqref{OC-candidate}, shows that 
\begin{equation}
\label{cond1}
0 \leq X^{x,y,\nu^{\star}}_t \leq d(\pi^y_t), \qquad t \geq 0,\,\,\mathbb{P}-a.s.
\end{equation}
Moreover, it is easy to see that we can express $\overline{\nu}^{\star}$ of \eqref{OC-candidate} as
\begin{equation}
\label{cond2}
\overline{\nu}^{\star}_t = \sup_{0\leq u \leq t} \Big(\frac{X^{x,y,0}_s - d(\pi^y_s)}{X^{1,y,0}_s}\Big) \vee 0, \qquad \overline{\nu}^{\star}_{0^-}=0.
\end{equation}

The previous equations allow us to make some remarks about the optimal debt management policy of our problem.
\begin{itemize}
\item[(i)] If at initial time the level of the debt ratio $x$ is above $d(y)$, then an immediate lump sum reduction of amplitude $(x-d(y))$ is optimal. 
\item[(ii)] At any $t\geq 0$, it is optimal to keep the debt ratio level below the belief-dependent ceiling $d$. 
\item[(iii)] If the level of the debt ratio at time $t$ is strictly below $d(\pi_t)$, there is no need for interventions. The government should intervene to reduce its debt only at those (random) times $t$ at which the debt ratio attempts to rise above $d(\pi_t)$. These interventions are then minimal, in the sense that $(X^{x,y,\nu^{\star}},\pi^y,\nu^{\star})$ solves a Skorokhod reflection problem at the free boundary $d$.
\item[(iv)] Recall that the debt ceiling $d$ is an increasing function of the government's belief that the economy is enjoying a phase of fast growth. Then, with regard to the previous description of the optimal debt reduction rule, we have that the more the government believes that the economy is in a good shape, the more the fiscal space is, and the less strict the optimal debt reduction policy should be. 
\end{itemize}


\subsection{Regularity of the Value Function of Problem (P3) and Related HJB Equation}
\label{sec:regularityV}

Combining the results collected so far, we are now able to prove that the value function $V$ of control Problem \textbf{(P3)} is a twice-continuously differentiable function. As a byproduct, $V$ is a classical solution to the corresponding Hamilton-Jacobi-Bellman (HJB) equation.

By Theorem \ref{teo:ver-2} we know that $V(x,y)= \int_{0}^{x}\frac{1}{z}v(z,y)dz$, for all $(x,y)\in \overline{\mathcal{O}}:=[0,\infty) \times [0,1]$. Hence, thanks to Theorem \ref{thm:C1} and to the dominated convergence theorem, we immediately obtain the following result.

\begin{lemma}
\label{lem:regV}
One has that $V \in C^1(\mathcal{O}) \cap C(\overline{\mathcal{O}})$. Moreover, $V_{xx}\in C(\mathcal{O})$, as well as $V_{xy}\in C(\mathcal{O})$.
\end{lemma}

To take care of the second derivative $V_{yy}$ we follow ideas used in \cite{DeA18}. In particular, we determine the second weak derivative of $V$ (recall that $V_{y}$ is continuous by Theorem \ref{thm:C1}), and we then show that it is a continuous function. This is accomplished in the next proposition.
\begin{proposition}
\label{prop:Vyy}
Let $\theta^2:=\frac{1}{2}[(\alpha_1-\alpha_2)^2 + \frac{(g_2-g_1)^2}{\sigma^2}]$. We have $V_{yy} \in C(\mathcal{O})$ with
\begin{align}
\label{defVyy}
& V_{yy}(x,y)= - \frac{1}{\theta^2 y^2(1-y)^2}\Big[\big(\beta_2 + (g_2-g_1)y - \frac{1}{2}\sigma^2\big)\big(v(x\wedge d(y),y)-v(0+,y)\big) \nonumber \\
& + h\big(x\wedge d(y)\big) + \frac{1}{2}\sigma^2\big(x\wedge d(y)\big) v_x(x\wedge d(y),y)\Big] + \frac{\big(\lambda_2 - (\lambda_1+\lambda_2)y\big)}{\theta^2 y^2(1-y)^2}\bigg(\int_{0}^{x \wedge d(y)} \frac{1}{z} v_y(z,y) dz\bigg)  \\
& - \frac{\rho}{\theta^2 y^2(1-y)^2}\bigg(\int_{0}^{x \wedge d(y)} \frac{1}{z} v(z,y) dz\bigg). \nonumber
\end{align}
\end{proposition}
\begin{proof}
Notice that $V_y(x,y)=\int_0^x\frac{1}{z}v_y(z,y)dz$, and therefore $V_y(x,\cdot)$ is a continuous function for all $x>0$ by Theorem \ref{thm:C1} (notice indeed that by the bounds in \eqref{ineq-final} and the multiplicative dependence of $\X^{z,y}$ with respect to $z$ one has that $\frac{1}{z}v_y(z,y)$ is integrable at zero). Hence, its weak derivative with respect to $y$ is a function $g \in L^1_{loc}(\mathcal{O})$ such that for any test function $\varphi \in C^{\infty}_c((0,1))$ one has
\begin{equation}
\label{weakderivative}
\int_0^1 V_y(x,y)\varphi'(y)dy = - \int_0^1 g(x,y)\varphi(y)dy.
\end{equation}
We now aim at evaluating $g$ and at showing that it coincides with the right-hand side of \eqref{defVyy}. 

Denote by $m(x)$, $x>0$, the generalized right-continuous inverse of $d(y)$, $y\in[0,1]$; that is, $m(x):= \inf\{y\in[0,1]:\,d(y) \geq x\}$. Then, noticing that $v_y=0$ on $\{(x,y) \in \mathcal{O}:\, x>d(y)\}$ and using Fubini's theorem, we can write
\begin{align}
\label{weak1}
& \int_0^1 V_y(x,y)\varphi'(y)dy = \int_0^1 \bigg(\int_0^{x \wedge d(y)} \frac{1}{z}v_y(z,y) dz\bigg) \varphi'(y)dy \nonumber \\
& = \int_0^x \frac{1}{z} \bigg(\int_{m(z)}^1 v_y(z,y) \varphi'(y)dy \bigg) dz  \\
& = \int_0^x \frac{1}{z} \Big[v_y(z,1)\varphi(1) - v_y(z,m(z))\varphi(m(z)) - \int_{m(z)}^1 v_{yy}(z,y) \varphi(y)dy \Big] dz \nonumber\\
&=  - \int_0^x \frac{1}{z} \bigg(\int_{m(z)}^1 v_{yy}(z,y) \varphi(y)dy\bigg) dz, \nonumber
\end{align}
where we have used that $v_y(z,m(z))=0$ for all $z\in (0,x)$, $x>0$, as well as $\varphi(1)=0$.

By Lemma \ref{lem:vC2cont} (cf.\ also \eqref{generator}), for any $y>m(z)$, for any $z\in (0,x)$ with $x>0$, we have that 
\begin{align}
\label{vyy}
& v_{yy}(z,y)= \frac{1}{\theta^2 y^2(1-y)^2}\Big[\rho v(z,y) - \big(\lambda_2 - (\lambda_1+\lambda_2)y\big) v_y(z,y) - zh'(z)
\nonumber \\
& - \frac{1}{2}\sigma^2 z^2 v_{xx}(z,y) - \big(\beta_2 + (g_2-g_1)y\big)z v_x(z,y)  \Big].
\end{align}
Inserting the latter expression in the last integral term on the right-hand side of \eqref{weak1}, using again Fubini's theorem and then integrating the derivatives with respect to $x$, we find
\begin{align}
\label{weak2}
& \int_0^1 V_y(x,y)\varphi'(y)dy = - \int_0^x \frac{1}{z} \bigg(\int_{m(z)}^1 v_{yy}(z,y) \varphi(y)dy\bigg) dz \nonumber \\
& = \int_0^1 \frac{\big(\lambda_2 - (\lambda_1+\lambda_2)y\big)}{\theta^2 y^2(1-y)^2}\bigg(\int_{0}^{x \wedge d(y)} \frac{1}{z} v_y(z,y) dz\bigg)\varphi(y) dy \nonumber \\
&- \int_0^1 \frac{\rho}{\theta^2 y^2(1-y)^2}\bigg(\int_{0}^{x \wedge d(y)} \frac{1}{z} v(z,y) dz\bigg)\varphi(y) dy \nonumber \\
& + \int_0^1 \Big[h\big(x\wedge d(y)\big) + \big(\beta_2 + (g_2-g_1)y\big)\big(v(x\wedge d(y),y)-v(0+,y)\big)  \\
& + \frac{1}{2}\sigma^2\big(x\wedge d(y)\big) v_x(x\wedge d(y),y) - \frac{1}{2}\sigma^2 \big(v(x\wedge d(y),y)-v(0+,y)\big)\Big] \frac{\varphi(y)}{\theta^2 y^2(1-y)^2} dy, \nonumber 
\end{align}
where we have also used that $h(0)=0$.
Finally, setting
\begin{align*}
& g(x,y):= - \frac{1}{\theta^2 y^2(1-y)^2}\Big[h\big(x\wedge d(y)\big) + \big(\beta_2 + (g_2-g_1)y - \frac{1}{2}\sigma^2\big)\big(v(x\wedge d(y),y)-v(0+,y)\big) \nonumber \\
& + \frac{1}{2}\sigma^2\big(x\wedge d(y)\big) v_x(x\wedge d(y),y)\Big] + \frac{\big(\lambda_2 - (\lambda_1+\lambda_2)y\big)}{\theta^2 y^2(1-y)^2}\bigg(\int_{0}^{x \wedge d(y)} \frac{1}{z} v_y(z,y) dz\bigg) \nonumber \\
& - \frac{\rho}{\theta^2 y^2(1-y)^2}\bigg(\int_{0}^{x \wedge d(y)} \frac{1}{z} v(z,y) dz\bigg), 
\end{align*}
we see that \eqref{weak2} reads $\int_0^1 V_y(x,y)\varphi'(y)dy = - \int_0^1 g(x,y)\varphi(y)dy$, so that $g$ identifies with the second weak derivative of $V$ with respect to $y$. Notice that $g$ is continuous by the continuity of $d$, $v$, $v_x$, $h$, and the fact that $\int_0^{x\wedge d(y)} \frac{1}{z}v(z,y)dz$ and $\int_0^{x\wedge d(y)} \frac{1}{z}v_y(z,y)dz$ are finite due to \eqref{sol:X}, \eqref{valueOS}, and \eqref{ineq-final}. The proof is therefore completed.
\end{proof}

Thanks to Lemma \ref{lem:regV} and Proposition \ref{prop:Vyy} we have that $V \in C^2(\mathcal{O}) \cap C(\overline{\mathcal{O}})$. As a byproduct of this, by the Dynamic Programming Principle and standard means based on an application of Dynkin's formula, we obtain the next result. 
\begin{proposition}
\label{prop:HJB}
Recall the second-order differential operator $\mathbb{L}$ defined in \eqref{generator}. The value function $V$ of Problem \textbf{(P3)} is a classical solution to the HJB equation
\begin{equation*}
\min\big\{\big(\mathbb{L}-\rho\big)V(x,y) + h(x), 1 - V_x(x,y)\big\}=0, \quad (x,y) \in \mathcal{O},
\end{equation*}
with boundary condition $V(0,y)=0$ for any $y\in [0,1]$. 
\end{proposition}


\section*{Acknowledgments}
\noindent The research of Claudia Ceci is partially supported by ``Gruppo Nazionale per l'Analisi Matematica, la Probabilit\`a e le loro Applicazioni'' (GNAMPA) of ``Istituto Nazionale di Alta Matematica'' (INdAM). Financial support by the German Research Foundation (DFG) through the Collaborative Research Centre 1283 is gratefully acknowledged by Giorgio Ferrari. We wish to thank Luciano Campi, Tiziano De Angelis, Paola Mannucci, Fabio Paronetto, Paavo Salminen, and Wolfgang Runggaldier for useful discussions.


\appendix

\section{Filtering Results}
\label{app:proofs}

\begin{proof}[\textbf{Proof of Proposition \ref{rappre}}]
Since the innovation processes $(I,I^1)$ (see \eqref{Inn}) and the random measure $m(dt,dq)$ (see \eqref{measure-m} and \eqref{fm}) are $\mathbb H$-adapted, then $\mathbb F^I \vee \mathbb F^{I^1}\vee \mathbb F^m \subseteq \mathbb H$. In general, the latter inclusion could be strict. Let us now consider the exponential $\mathbb F$-martingale solving
\begin{equation}
dL_t = - L_t  \Big\{ \frac{\beta(Z_t)}{\sigma} dW_t + \alpha(\eta_t, Z_t) dB_t \Big \}, \quad t\geq 0,
\end{equation}
and define a probability measure $\Q$ on $(\Omega,\mathcal F)$, equivalent to $ \P$ on $\mathcal{F}_t$, and such that 
$$\frac{d\mathbb Q}{d\P}\Big|_{\mathcal F_t} = L_t, \quad t\geq 0.$$
Notice that Assumption \eqref{ass:Novikov} ensures that $L$ is indeed an $\mathbb{F}$-martingale. By Girsanov's theorem, the processes
\begin{equation} \label{Inn1}
\widetilde W_t := W_t + \int_0^t \frac{\beta(Z_s)}{\sigma} ds, \quad  \widetilde B_t := B_t + \int_0^t \alpha(\eta_s, Z_s) ds, \quad t \geq 0,
\end{equation}
are $(\Q, \mathbb F)$-independent Brownian motions. We now prove that $\mathbb F^{\widetilde W} \vee \mathbb F^{\widetilde B}\vee \mathbb F^m = \mathbb H.$ On the one hand, the inclusion $\mathbb F^{\widetilde W} \vee \mathbb F^{\widetilde B}\vee \mathbb F^m \subseteq \mathbb H$ follows from the fact that ${\widetilde W}$ and ${\widetilde B}$ turn out to be $ \mathbb H$-adapted since they can be written as
\begin{equation} \label{Inn2}
\widetilde W_t = I_t + \int_0^t \frac{\pi_s(\beta)}{\sigma} ds, \quad  \widetilde B_t := I^1_t + \int_0^t \pi_s(\alpha(\eta_s, \cdot )) ds,  \quad t \geq 0.
\end{equation}
To prove the converse, let us observe that, under the probability measure $\Q$, the process $X^0$ and $\eta$ solve the following stochastic differential equations  
$$ dX^0_t = X^0_t \sigma d\widetilde W_t, \quad X_0^0=x >0,$$
\begin{equation}\label{etaQ}
d\eta_t = \sigma_1(\eta_t) d \widetilde W_t + \sigma_2(\eta_t) d \widetilde B_t  +  \int_\mathbb R q m(dt, dq), \quad \eta_0=q \in  \I,
\end{equation}
respectively. Clearly $X^0$ is $\mathbb F^{\widetilde W}$-adapted. Recalling \eqref{measure-m1}, the solution to equation  \eqref{etaQ} can be constructed iteratively. More precisely, $\forall t \in [0, T_1)$, the process $\eta$ solves 
$$d\eta_t =   \sigma_1(\eta_t) d \widetilde W_t + \sigma_2(\eta_t) d \widetilde B_t, \quad \eta_0=q \in  \I,$$
and for any time between two consecutive jump times, i.e.\ $t \in [T_n, T_{n+1})$, $n\geq 1$, one has 
$$d\eta_t =   \sigma_1(\eta_t) d \widetilde W_t + \sigma_2(\eta_t) d \widetilde B_t, \quad  \eta_{T_n} = \eta_{T_n-} + \zeta_n.$$

By Assumption \ref{ass:eta}, this sequence stochastic differential equations has a unique strong solution on any interval $[T_n, T_{n+1})$, and this in turn gives the unique strong solution $\eta$ to \eqref{etaQ}. Moreover, $\eta$ turns out to be  $F^{\widetilde W} \vee \mathbb F^{\widetilde B}\vee \mathbb F^m $-adapted.

Then, by applying Corollary III.4.3.1 in \cite{JacodShir}, we have that every $(\Q, \mathbb H)$-local martingale $\widetilde M$ admits the decomposition 
$$\widetilde M_t = \widetilde M_0 + \int_0^t \widetilde \varphi_s d\widetilde W_s + \int_0^t \widetilde \psi_s d\widetilde B_s + \int_0^t \int_{\mathbb R} \widetilde w(s,q) m^{\pi}(dt, dq),$$
where $\widetilde \varphi$ and $\widetilde \psi$ are $\mathbb H$-adapted processes, and $\widetilde w$ is an $\mathbb H$-predictable process indexed by $\mathbb R$, such that for all $t \geq 0$
$$\int_0^t \widetilde \varphi^2_s ds < \infty, \quad \int_0^t  \widetilde \psi^2_s ds < \infty, \quad \int_0^t  \int_{\mathbb R} |\widetilde w(s,q)| m^{p, \mathbb H}(dt, dq) < \infty \quad \Q-a.s.$$ 

Let now $M$ be a $(\P, \mathbb H)$-local martingale, then $\widetilde M :=  M \widetilde L^{-1}$ is a $(\Q, \mathbb H)$-local martingale, where
$$\widetilde L_t := \E[ L_t | \mathcal H_t] = \frac{d\mathbb Q}{d\P}\Big|_{\mathcal H_t}, \quad t\geq 0.$$
Taking into account \eqref{Inn2}, we have that $\widetilde L$ solves
$$d\widetilde L_t = - \widetilde L_t  \left \{ {\frac{\pi_t(\beta)}{\sigma}} dI_t + \pi_t(\alpha(\eta_t,\cdot )) dI^1_t  \right \}, \quad  
\quad \widetilde L_0=1,$$
and by applying the product formula to $M= \widetilde M \widetilde L$, we easily obtain that
\begin{align*}
& dM_t = \widetilde M_{t^-} d\widetilde L_t + \widetilde L_{t} d\widetilde M_t + d \langle \widetilde M^c, \widetilde L^c \rangle_t \nonumber \\
& = (\widetilde L_t \widetilde \varphi_t - M_t  \frac{\pi_t(\beta)}{\sigma} ) dI_t + (\widetilde L_t \widetilde \psi_t - M_t  \pi_t(\alpha(\eta_t,\cdot )) ) dI^1_t + \int_{\mathbb R} \widetilde w(t,q) \widetilde L_t m^\pi(dt, dq).
\end{align*}

To conclude, we thus only need to set
$$\varphi_t:= \widetilde L_t \widetilde \varphi_t - M_t  \frac{\pi_t(\beta)}{\sigma}, \quad \psi_t:=\widetilde L_t \widetilde \psi_t - M_t  \pi_t(\alpha(\eta_t,\cdot )), \quad w(t,q) := \widetilde w(t,q) \widetilde L_t.$$
\end{proof}


\begin{proof}[\textbf{Proof of Theorem \ref{filtering:eq}}]

In order to derive the filtering equation solved by ${\underline \pi}_t = (\pi_t(i); i\in S)_{t \geq 0}$, we apply the innovation approach (see, for instance, Chapter IV in \cite{Bremaud}). In this proof we shall use the two well-known facts: 
\begin{itemize}
\item[(i)] for every $\mathbb{F}$-martingale $m$, the projection over $\mathbb{H}$ is an $\mathbb{H}$-martingale; that is, $\hat m_t := \E[ m_t| \mathcal H_t]$, $t \geq 0$, is an $\mathbb{H}$-martingale; 

\item[(ii)] for any $\mathbb{F}$-progressively measurable and integrable process $\Psi$ we have that 
$$\E\bigg[ \int_0^t \Psi_s ds \Big| \mathcal H_t\bigg] -  \int_0^t \E\big[ \Psi_s  \big| \mathcal H_s\big] ds$$
is an $\mathbb{H}$-martingale.
\end{itemize}

The first step of the innovation method consists in writing the process $ \mathds{1}_{\{ Z_t = i \}}$, $i\in S$, as a semimartingale. Denoting by $L^Z$ the Markov generator of the state process $Z$, we have that 
$$L^Z f_i(j) = \sum_{k \in  S} \lambda_{ki} f_k(j), \quad i,j \in S,$$
where $f_k(j)= \mathds{1}_{\{j=k\}}$. Hence, for any $i \in S$, we can write
$$ \mathds{1}_{\{ Z_t = i \}} = f_i(Z_t) = f_i(Z_0)  + \int_0^t  L^Z f_i(Z_s) ds + m_t(i),$$
where $(m_t(i))_{t \geq 0}$ is an $\mathbb{F}$-martingale. By taking the conditional expectation with respect to $\mathcal H_t$, and using (i) and (ii) above, we obtain that
\begin{equation} 
\label{C1} 
\pi_t(i) = y_i + \int_0^t   \sum_{k \in  S} \lambda_{ki} \pi_s(k) ds + M_t(i),
\end{equation}
where $M(i)$ is an $\mathbb{H}$-martingale null at zero. 
Proposition \ref{rappre} ensures the existence of processes $\psi(i)$ and $\varphi(i)$ that are $\mathbb H$-predictable, and $w_i$ which is $\mathbb H$-predictable and indexed by $\mathbb R$, such that 
\begin{equation} 
\label{C2}
M_t(i) =  \int_0^t \psi_s(i) dI_s + \int_0^t \varphi_s(i) dI^1 _s + \int_0^t \int_{\mathbb R} w_i(s,q) m^{\pi}(ds, dq),
\end{equation}
with 
$\E[\int_0^t \varphi^2_s(i) ds] < \infty$, $\E[ \int_0^t \psi^2_s(i) ds ]< \infty$ and $\E[ \int_0^t  \int_{\mathbb R} |w_i(s,q)| m^{p, \mathbb H}(dt, dq)] < \infty$, $ t\geq 0$. 
To obtain equation \eqref{KS} it only remains to prove that 
$$\psi_s(i) = \pi_s(i) \sigma^{-1}  \Big\{ \beta_i -  \sum_{j=1}^Q  \beta_j \pi_s(j) \Big\}  , \quad \varphi_s(i)= \pi_s(i)\Big\{ \alpha(\eta_s, i) -  \sum_{j=1}^Q  \alpha(\eta_s, j) \pi_s(j) \Big\}$$
$$w_i(s,q) = w^\pi_i(s,q) - \pi_{s^-}(i)$$
with $w^\pi_i$ given in \eqref{salto}.
 
Following the same lines of the proof of Theorem 3.1 in \cite{ColaneriCeci2012} we  can derive the structure of the processes $\psi(i)$, $\varphi(i)$ by imposing the following equalities
$$\forall i \in S, \quad  \E[  f_i(Z_t) \widetilde W_t | \mathcal H_t] = \pi_t(i) \widetilde W_t,  \quad \E[  f_i(Z_t) \widetilde B_t | \mathcal H_t] = \pi_t(i) \widetilde B_t,$$ 
where $\widetilde W$ and $\widetilde B$ are the $\mathbb{H}$-Brownian motions defined in \eqref{Inn1}. To derive the expression of $w_i$, we consider a bounded process $\Gamma$ of the form $\Gamma_t = \int_0^t \int_{\mathbb R} \gamma(s,q) m(ds, dq)$, with $\gamma$ $\mathbb{H}$-predictable process indexed by ${\mathbb R}$. Since $\Gamma$ is $\mathbb{H}$-adapted, the equality
\begin{equation} \label{conti}
\forall i \in S, \quad \E[  f_i(Z_t) \Gamma_t | \mathcal H_t] = \pi_t(i) \Gamma_t \end{equation}
holds.
By applying the product rule (taking into account no common jumps between $Z$ and $N$) we obtain
$$d(f_i(Z_t) \Gamma_t )= f_i(Z_{t^-}) d\Gamma_t + \Gamma_{t^-} df_i(Z_t) =  \Gamma_{t^-} L^Z f_i(Z_t) dt + 
\int_{\mathbb R} f_i(Z_{t^-}) \gamma(t,q) m^{p, \mathbb{F}}(dt, dq) + \mathcal{M}^{\mathbb{F}}_t, $$ 
where $m^{p, \mathbb{F}}(dt, dq)$ is the $\mathbb{F}$-dual predictable projection of $m(dt,dq)$ given in \eqref{mp}, and $\mathcal{M}^{\mathbb{F}}$ is an $\mathbb{F}$-martingale.
By projection onto $\mathcal H_t$, and denoting by $\mathcal{M}^{\mathbb{H}}$ an $\mathbb{H}$-martingale, we have that
\begin{equation} \label{C3} d \E[f_i(Z_t) \Gamma_t | \mathcal H_t] = \Gamma_{t} \pi_{t}(L^Z f_i) dt  +  \lambda^N(i) \pi_{t^-}(i)  \gamma(t, c(\eta_{t^-}, i))  \mathds{1}_{\{  c(\eta_{t^-}, i)\neq 0 \}} dt + \mathcal{M}^{\mathbb{H}}_t. 
\end{equation}

On the other hand, the product rule and  \eqref{C1} and \eqref{C2} yield 
\begin{align*}
& d(\pi_{t}(i) \Gamma_t )= \pi_{t^-}(i) d\Gamma_t + \Gamma_{t^-} d\pi_t(i) + d \langle \pi(i), \Gamma \rangle_t \nonumber \\
& = \Gamma_{t} \pi_{t}(L^Z f_i) dt + \Gamma_{t^-} dM_t(i) + \int_{\mathbb R} \gamma(t,q) w_i(t,q)m(dt, dq).
\end{align*} 

Recalling that $m^{p, \mathbb{H}}(dt, dq)$ is the $\mathbb{H}$-dual predictable projection of $m(dt,dq)$ given in \eqref{dual}, we find
\begin{equation} 
\label{C4}
d(\pi_{t}(i) \Gamma_t )= \Gamma_{t} \pi_{t}(L^Z f_i) dt + \int_{\mathbb R} \gamma(t,q) w_i(t,q)m^{p, \mathbb{H}}(dt, dq) + \mathcal{M}^{\mathbb{H}}_t, 
\end{equation}
where, again, $\mathcal{M}^{\mathbb{H}}$ is an $\mathbb{H}$-martingale. 

Gathering equations \eqref{conti}, \eqref{C3}, and \eqref{C4}, we obtain that for a.e.\ $t\geq 0$
\begin{align*}
&\lambda^N(i) \pi_{t^-}(i)  \gamma(t, c(\eta_{t^-}, i))  \mathds{1}_{\{  c(\eta_{t^-}, i)\neq 0 \}} \\
& =  \sum_{j=1}^Q\pi_{t^-}(j) \lambda^N(j) \gamma(t, c(\eta_{t^-}, j)) \mathds{1}_{\{ c(\eta_{t^-}, j) \neq 0\}} ( \pi_{t^-}(i) + w_i(t, c(\eta_{t^-}, j))).
\end{align*} 

Choose now $\gamma(t,q)$ of the form $\gamma(t,q) = C_t \mathds{1}_{A}(q) \mathds{1}_{t \leq T_n}$, with $C$ any bounded, $\mathbb{H}$-predictable, positive process and $A \in {\mathcal  B}(\mathbb R)$. Observe that $\Gamma$ is bounded since $|\Gamma_t| \leq \int_0^{t \wedge T_n} C_s dN_s \leq D n$, with $D$ a positive constant. Then the following equality holds on $\{t \leq T_n\}$ 
$$\forall \in  {\mathcal  B}(\mathbb R), \quad \int_A \nu_t(i,dq) = \int_A (\pi_{t^-}(i) + w_i(t,q)) \nu_t(dq),$$
where we have set
$$\nu_t(i,dq):= \lambda^N(i)\pi_{t^-}(i) \mathds{1}_{\{ c(\eta_{t^-}, i) \neq 0\}} \delta_{c(\eta_{t^-}, i)}(dq), \quad \nu_t(dq) := \sum_{i=1}^Q \nu_t(i,dq).$$
Thus, on $\{t \leq T_n\}$, 
$$w^\pi_i(t,q) =  w_i(t,q) - \pi_{t^-}(i) =  \frac{d \nu_t(i,dq)}{\nu_t(dq)}, \quad \forall i \in S.$$
Finally, since the counting process $N$ is nonexplosive, $T_n \uparrow \infty$ a.s.\ for $n \uparrow \infty$, and this yields \eqref{salto}.
\end{proof}


\begin{proof}[\textbf{Proof of Proposition \ref{uniq}}]
By Proposition \ref{recursive}, equations \eqref{KS} and \eqref{eta1} are equivalent to a system of recursive equations between consecutive jump times, i.e.\ for $t \in [T_n, T_{n+1})$, $n=0,1,\dots$
\begin{align*} 
 \pi_t(i) =   \pi_{T_n}(i) + \int_{T_n}^t  b^{\pi}( \underline \pi_s, \eta_s, i) ds + \int_{T_n}^t \sigma^{\pi}_ 1( \underline \pi_s, i) dI_s + \int_{T_n}^t  \sigma^{\pi}_ 2( \underline \pi_s, i)  dI^1_s,  \quad  i\in S,
\end{align*}

$$\eta_t  =   \eta_{T_n} +  \int_{T_n}^t  b^{\eta} ( \underline \pi_s, \eta_s) ds + \int_{T_n}^t \sigma_1(\eta_s) dI_s + \int_{T_n}^t \sigma_2( \eta_s) dI^1_s,$$

where we have set 
$$b^{\pi}(\underline y , q, i):= \sum_{j=1}^Q \lambda_{ji}  y_j -   y_i\Big[ \lambda^N(i) \mathds{1}_{\{ c( q, i) \neq 0\}}  - \sum_{j=1}^Q   \lambda^N(j) y_j \mathds{1}_{\{ c( q, j) \neq 0\}} \Big], \quad b^\eta (\underline y , q):= \sum_{j=1}^Q  y_j b_1(q, j),$$
$$\sigma^{\pi}_ 1( \underline y, i):= \sigma^{-1}  y_i \Big\{ \beta_i  -  \sum_{j=1}^Q  \beta_j  y_j \Big\}, \quad \sigma^{\pi}_ 2( \underline y, i):= y_i \Big\{ \alpha(q, i)  -  \sum_{j=1}^Q  \alpha( q, j) y_j \Big\},$$
with the update at time $T_n$ given by

\begin{equation} 
\label{upd}
 \pi_{T_n}(i) = \frac{ \lambda^N(i)  \pi_{T_n^-}(i) \mathds{1}_{\{ \zeta_n = c(\eta_{T_n^-}, i) \}}}{\sum_{j=1}^Q\lambda^N(j) \pi_{T_n^-}(j) \mathds{1}_{\{ \zeta_n = c(\eta_{T_n^-}, j) \}}}, \quad  i \in S, \qquad  \eta_{T_n} = \eta_{T_n^-} + \zeta_n. \end{equation}

Recall that, by assumption, the function $\alpha(q, i)$ given in \eqref{alfa}
is locally-Lipschitz with respect to $q$ and satisfies a (global) sublinear growth condition with respect to $q \in  \I$, uniformly in $i \in S$. 

We now develop the proof of uniqueness by distinguishing among three different cases related to the jumps' amplitude $c$; namely, $c\neq0$, $c=0$, and $c\in \mathbb{R}$.

In the case $c\neq0$, we have that $b^{\pi}(\underline y , q, i):= \sum_{j=1}^Q \lambda_{ji}  y_j -   y_i\Big[ \lambda^N(i)  - \sum_{j=1}^Q   \lambda^N(j) y_j  \Big],$ and it is easy to verify that between two consecutive jump times the pair $(\underline {\pi}, \eta)$ solves a $(Q+1)$-dimensional stochastic differential equation with coefficients satisfying locally-Lipschitz conditions and (global) sublinear growth conditions with respect to $(\underline y, q) \in \mathcal{Y} \times \mathbb R$, uniformly in $i \in S$. As a consequence, strong uniqueness holds between two consecutive jump times; i.e.\ for $t\in[T_{n-1}, T_{n})$, $n=1,\dots$.  Moreover, since the update at jump time $T_n$ (see \eqref{upd}) depends on the process $(\underline {\pi}_t, \eta_t)$ for $t \in [T_{n-1}, T_{n})$, we have strong uniqueness of the solution to system \eqref{KS} and \eqref{eta1} for all $t\geq0$. 

In the case $c=0$, equations \eqref{KS} and \eqref{eta1} reduce to 
\begin{align*} 
d\pi_t(i) =   b^{\pi}( \underline \pi_t, \eta_t, i) dt +\sigma^{\pi}_ 1( \underline \pi_t, i) dI_t + \sigma^{\pi}_ 2( \underline \pi_t, i)  dI^1_t,  \quad  i\in S, \quad t \geq 0,
\end{align*}

$$d\eta_t  =   b^{\eta} ( \underline \pi_t, \eta_t) dt + \sigma_1(\eta_t) dI_t + \sigma_2( \eta_t) dI^1_t, \quad t \geq 0,$$
where, in particular, $b^{\pi}(\underline y , q, i)= \sum_{j=1}^Q \lambda_{ji}  y_j$. It is easy to check that also in this case strong uniqueness follows by the locally-Lipschitz property of the coefficients and by their (global) sublinear growth condition. 

In the case $c\in \mathbb{R}$ the jumps' amplitude can assume any possible real value. In particular, $c$ can be such that $\eta$ and $N$ do not have only common jumps: $N$ might jump at a time at which $c(\eta_{t^-},Z_{t^-}) =0$, so that $\eta$ does not jump at that time. The treatment of this case is more delicate and should be performed separately. Indeed, the uniqueness cannot be proved by using the arguments employed in the previous two cases because of the presence of $\mathds{1}_{\{ c( q, i) \neq 0\}}$ in the coefficient $b^{\pi}$ which prevents to prove Lipschitz-continuity of $b^{\pi}$ with respect to $q$. However, one might prove uniqueness by relying on the filtered martingale problem associated to the infinitesimal generator of the triplet $(Z,X^0,\eta)$. We refer to the seminal paper \cite{KurtzOcone} and to Theorem 3.3 and Appendix B in the more recent \cite{ColaneriCeci2012}.

\end{proof}

\renewcommand{\theequation}{A-\arabic{equation}}


\end{document}